\documentclass{amsart}
\usepackage{placeins}
\usepackage{amsmath,amssymb}
\usepackage{textcomp}
\usepackage{stmaryrd}
\usepackage{amsmath}
\usepackage{enumitem}
\usepackage{makecell}
\usepackage{enumerate}
\usepackage{blindtext}
\extrafloats{100}
\usepackage{booktabs}

\usepackage[caption=false]{subfig}

\usepackage{algorithm,algcompatible,amsmath}
\algnewcommand\OUTER{\item[\textbf{OUTER CYCLE}]}%
\algnewcommand\INNER{\item[\textbf{INNER CYCLE:}]}%
\usepackage{graphicx}
\usepackage{tabularx,ragged2e,booktabs,caption}
\newcolumntype{C}[1]{>{\Centering}m{#1}}

\usepackage{geometry}
\usepackage{subfloat}
\DeclareMathOperator*{\argmin}{argmin}

\numberwithin{equation}{section}
\newtheorem*{step1}{Step $(i)$}
\newtheorem*{step2}{Step $(ii)$}
\newtheorem*{step3}{Step $(iii)$}
\theoremstyle{plain}
\usepackage{cleveref}
  \newtheorem{theorem}{Theorem}[section]

\theoremstyle{remark}

\theoremstyle{defn}
\newtheorem{proposition}[theorem]{Proposition}

\newtheorem{corollary}[theorem]{Corollary}
\newtheorem{remark}{Remark}[section]

\def\to{\rightarrow}
\def\eps{\varepsilon}

\newcommand{\R}{{\mathbb R}}

\title[]{On the monotone and primal-dual active set schemes for $\ell^p$-type problems, $p \in (0,1]$}
\thanks{This work was  supported by the ERC advanced grant $668998$ (OCLOC) under the EU's H$2020$ research programme. }

\author{Daria Ghilli}
 \address{Daria Ghilli, University of Graz, Institute of Mathematics and Scientific Computing,  Universit{\"a}tsplatz 3, Austria}
\email{daria.ghilli@uni-graz.at}

\author{Karl Kunisch}
\address{Karl Kunisch, University of Graz, Institute of Mathematics and Scientific Computing,  Universit{\"a}tsplatz 3, Austria \\
Johann Radon Institute for Computational and Applied Mathematics (RICAM), Austrian Academy of Sciences, Altenbergerstrasse 69, Linz, Austria\\}
\email{karl.kunisch@uni-graz.at}

\begin{document}

\maketitle

\begin{abstract}
 Nonsmooth nonconvex optimization problems  involving the
 $\ell^p$ quasi-norm, $p \in (0, 1]$, of a   linear map are considered. A monotonically convergent
scheme for a regularized version of the original problem is developped and necessary optimality
conditions for the original prolem in the form of a complementary system amenable for computation are given. Then an algorithm for solving the above mentioned necessary optimality conditions is proposed. It is based on a combination of the monotone
scheme and a primal-dual active set strategy. The performance of the two algorithms is studied by means of a series of numerical tests in different cases, including   optimal control problems, fracture mechanics and microscopy image reconstruction. \\
\vspace{0.2cm}

{\bf keywords:} nonsmooth nonconvex optimization \and active-set method \and monotone algorithm \and optimal control problems \and image reconstruction \and fracture mechanics.\\
\vspace{0.1cm}

{\bf math. subclass:} 49K99, 49M05,65K10
\end{abstract}

\section{Introduction}
We consider the following nonconvex nonsmooth optimization problem
\begin{equation}\label{optprob2mi}
\min_{x \in \R^n}J(x)=\frac{1}{2}|A x-b|_2^2+\beta |\Lambda x|^p_p,
\end{equation}
where $A \in \mathbb{M}^{m\times n}$, $\Lambda \in \mathbb{M}^{r\times n},$ $b \in \R^m, p \in (0,1]$ and $\beta \in \R^+$.
Here 
$$
|x|_p=\left(\sum_{k=1}^n |x_k|^p\right)^{\frac{1}{p}},
$$
which is a norm for $p=1$ and a quasi-norm for $0<p<1$. \\
Optimization of problems as \ref{optprob2mi}  arises frenquently in many applications as an efficient way to extract the essential
features of generalized solutions. In particular, many problems in sparse learning and compressed sensing can be written as \ref{optprob2mi} with $\Lambda=I$, $I$ being the identity (see e.g. \cite{CH09, NIZHZH08}   and the references therein). In image analysis, $\ell^p$-regularisers as in \ref{optprob2mi}   have recently been  proposed    as nonconvex extensions of the total generalized variation (TGV)  regularizer used to reconstruct piecewise smooth functions (e.g. in \cite{OHDOBRPO15, HIWU13}). Also, the use of $\ell^p$-functionals with $p\in (0,1)$ is of particular importance in fracture mechanics (see \cite{PI13}). Recently, sparsity techniques have been  investigated  also by the optimal control community, see e.g. \cite{CCK, HSW, St, KKZ, KI}. The literature on sparsity optimization problems as \ref{optprob2mi} is rapidly increasing, here we mention also \cite{BRLORE15, RAZA12, FOWA10, AFS}. \\
The nonsmoothness and nonconvexity  make the study of problems as \ref{optprob2mi} both an analytical and a numerical challenge. 
 Many numerical techniques have been developped when $\Lambda=I$  (e.g. in \cite{KI, GK, JIJILU14, JIJILURE13}) and attention has recently been given to the case of more general operators, here we mention e.g. \cite{OHDOBRPO15, HIWU13, LI} and we refer to the end of the introduction for further details. However, the presence of the matrix inside the $\ell^p$-term combined with the nonconvexity and nonsmoothness remains one main issue in the developments of numerical schemes for \ref{optprob2mi}.\\
In the present work, we first propose a monotone algorithm to solve a regularized version of \ref{optprob2mi}. The scheme is based on an iterative procedure solving a modified problem where the singularity at the origin is regularized. The convergence of this algorithm and the monotone decay of the cost during the iterations are proved. Then its performance is successfully tested in four different situations,  a time-dependent control problem, a fracture mechanic example for cohesive fracture models, an M-matrix example, and an elliptic control problem. \\
We also focus on the investigation of   suitable necessary optimality conditions for solving the original problem.   Relying on an augmented Lagrangian formulation, optimality conditions  of complementary type are derived. For this purpose we consider the case where $\Lambda$ is a regular matrix, since in the general case the optimality conditions of complementary type are not readily obtainable. 
An active set primal-dual strategy which exploits the particular form of these optimality conditions is developped. A new particular feature of our method  is that at each iteration level the monotone   scheme is used in order to solve the nonlinear equation satisfied by the non zero components. 
The convergence of the active set primal-dual strategy is proved in the case $\Lambda=I$ under a diagonal dominance condition. 
Finally the algorithm was tested on the same time-dependent control problem as the one analysed for the monotone scheme as well as for a miscroscopy image recontruction example. In all the above mentioned examples the matrix inside the $\ell^p$-term appears as a discretized gradient with very different purposes, e.g. as a regularization term in  imaging  and with modelling purposes in fracture mechanics. \\
 Similar type of algorithms  were proposed in \cite{KI} and \cite{GK} for problems as  \ref{optprob2mi}  in case of no matrix inside the $\ell^p$-term and in the infinite dimensional sequence spaces $\ell^p$, with $p \in [0,1]$. Our monotone and primal-dual active set monotone algorithm are inspired  by  the schemes studied respectively in \cite{KI} and \cite{GK}, but with the main novelties that now we treat the case of a regular matrix in the $\ell^p$-term and  we provide diverse numerical tests for both the schemes. Moreover, we prove the convergence of the primal-dual active set strategy. Note also that the monotone scheme has not been tested in the earlier papers.  \\
Let us recall some further literature concerning $\ell^p$, $p \in (0,1]$ sparse regularizers. Iteratively reweighted least-squares algorithms with suitable smoothing of the singularity at the origin were analysed in \cite{CY, LW, LXY}. In \cite{L} a unified convergence analysis was given and new variants were also proposed. An iteratively reweighted $\ell_1$ algorithm (\cite{CWB}) was developped in \cite{CZ} for a class of nonconvex $\ell^2$-$\ell^p$ problems, with $p \in (0,1)$.
A  generalized gradient projection method  for a general class of nonsmooth
non-convex functionals and  a generalized iterated shrinkage algorithm are analysed respectively in \cite{BRLORE15} and in \cite{ZMZFZ}. Also, in \cite{RAZA12} a surrogate functional approach combined with a gradient technique is proposed. 
 However, all the previous works do not investigate the case of a linear operator inside the $\ell^p$-term. \\
Then in \cite{OHDOBRPO15} an iteratively reweighted convex majorization algorithm is proposed for a  class of nonconvex problems including the $\ell^p$, $p \in (0,1]$ regularizer acting on a linear map.  However, an additional assumption of  Lipschitz continuity of the objective functional is required to establish convergence of the whole sequence generated by the algorithm. Nonconvex $TV^p$-models with $p \in (0,1)$ for image restoration are studied in \cite{HIWU13} by a Newton-type solution algorithm for a  regularized version of the original problem. \\
We mention  also \cite{JIJILURE13}, where  a primal-dual active set method is studied for problems as in \ref{optprob2mi} with $\Lambda=I$ for a large class of penalties including also the $\ell^p$, with $p \in [0,1)$. A continuation strategy with the respect to the regularization parameter $\beta$  is proposed and the convergence of the primal-dual active set strategy coupled with the continuation strategy is  proved. However,  in \cite{JIJILURE13}, differently from the present work, the nonlinear problem arising at each iteration level of the active set scheme is not investigated. Moreover, in \cite{JIJILURE13} the matrix $A$ has normalized column vectors, whereas in the present work $A$ is a general matrix.  \\
Finally, in \cite{LI} an alternating direction method of multipliers (ADMM) is studied in the case of a regular matrix inside the $\ell^p$-term, optimality conditions were derived and convergence was proved. Although the ADMM in \cite{LI} is also deduced from an augmented Lagrangian formulation,  we remark that the optimality conditions of that paper are of a different nature than ours  and hence the two approaches cannot readily be compared. We refer to  \ref{remc} for a more detailed explanation.\\
Concerning the general importance of $\ell^p$-functionals with $p \in (0,1)$, numerical experience has shown that their use can  promote sparsity  better than the $\ell^1$-norm (see \cite{6, 18, 45}), e.g. allowing possibly a smaller number of measurements in feature selection and compressed sensing (see also \cite{38, CH09, 8}). Moreover, many works demonstrated empirically that nonconvex regularization terms in total variation-based image restoration provide better edge preservation than the $\ell^1$-regularization (see \cite{35, 38, BR96, RB09}).  Also, the use of nonconvex optimization can be considered from natural image statistics \cite{HM99} and it appears to be more robust with respect to heavy-tailed distributed noise (see e.g. \cite{14}). \\
The paper is organized as follows. In  \ref{sec:optcond} we present our proposed monotone algorithm and we prove its convergence. In  \ref{numericsDeps} we report our numerical results for the four test cases mentioned above. In  \ref{alg} we derive the necessary optimality conditions for \ref{optprob2mi}, we describe our primal-dual active set strategy and  prove  convergence in the case $\Lambda=I$. Finally in  \ref{numericsD} we  report the numerical results obtained by testing the active set monotone algorithm in the two situations mentioned above.

 \section{Existence and monotone algorithm for a regularized problem}\label{sec:optcond}
For convenience of exposition, we recall the  problem under consideration
\begin{equation}\label{optprob2m}
\min_{x \in \R^n}J(x)=\frac{1}{2}|A x-b|_2^2+\beta |\Lambda x|^p_p,
\end{equation}
where $A \in \mathbb{M}^{m\times n}$, $\Lambda \in \mathbb{M}^{r\times n},$ $b \in \R^m, p \in (0,1]$ and $\beta \in \R^+$. \\
Throughout this section we assume 
\begin{equation}\label{asslambdi}
\mbox{Ker}(A)\cap\mbox{Ker}(\Lambda)=\{0\}.
\end{equation} 
The first result is existence for \ref{optprob2m}.
\begin{theorem}\label{2mex}
For any $\beta>0$, there exists a solution to \ref{optprob2m}.
\end{theorem}
\begin{proof}
Since $J$ is bounded from below, existence will follow from the continuity and coercivity of $J$.  Thus we prove that $J$ is coercive, that is, $|J(x_k)|_2 \to + \infty$ whenever $|x_k|_2\to + \infty$ for some sequence $\{x_k\} \subset \R^n$. By contradiction, suppose that $|x_k|_2\to + \infty$ and $J(x_k)$ is bounded. For each $k$, let $x_k=t_kz_k$ be such that $t_k\geq 0, x_k \in \R^n$ and $|z_k|_2=1. $ Since $t_k \to + \infty$, $p<2$, we have for $k$ sufficiently large 
$$
0\leq \frac{1}{2 t_k^2}|A x_k|_2^2+\beta \frac{1}{t_k^p}|\Lambda x_k|_p^p\leq (\frac{1}{2}+\beta)\frac{1}{t_k^p}\left( |A x_k|_2^2+|\Lambda x_k|_p^p\right) \to 0
$$
and hence
$$
\lim_{k \to + \infty} \frac{1}{2}|Az_k|_2^2+\beta|\Lambda z_k|^p_p=0.
$$
By compactness, the sequence $\{z_k\}$ has an accumulation point $\bar z$ such that $|\bar z|=1$ and $\bar z \in \mbox{Ker}(A)\cap \mbox{Ker}(\Lambda)$, which contradicts \ref{asslambdi}. 
\end{proof}
Following \cite{KI}, in order to overcome the singularity of $(|s|^p)'=\frac{ps}{|s|^{2-p}}$ near $s=0$, we consider for $\eps>0$ the following regularized version of \ref{optprob2m} 
\begin{equation}
\label{optprobeps2m}
\min_{x \in \R^n}J_\eps(x) = \frac{1}{2}|Ax-b|_2^2+\beta \Psi_\eps(|\Lambda x|^2), 
\end{equation}
where for $t \geq 0$
\begin{equation}\label{psieps}
\Psi_{\eps}(t)= \left\{
\begin{array}{ll}
\frac{p}{2}\frac{t}{\eps^{2-p}}+(1-\frac{p}{2})\eps^p \quad &\mbox{for }\,\, 0\leq t \leq \eps^2\\
\noalign{\smallskip}
t^{\frac{p}{2}} \quad & \mbox{ for }\,\, t \geq \eps^2,
\end{array}
\right.\,
\end{equation}
and $\Psi_\eps(|\Lambda x|^2)$ is short for $\sum_{i=1}^\infty \Psi_\eps(|(\Lambda x)_i|^2)$. \\
\begin{remark}
Notice that by the coercivity of the functional $J$ in \ref{optprob2m},  the coercivity of $J_\eps$ and hence existence for \ref{optprobeps2m} follow as well. 
\end{remark}
 The necessary optimality condition for \ref{optprobeps2m} is given by
$$
A^*Ax+\Lambda^*\frac{\beta p}{\max(\eps^{2-p},|\Lambda x|^{2-p})}\Lambda x=A^*b,
$$
where the max-operation is interpreted coordinate-wise.
\\
We set $y=\Lambda x$. Then
\begin{equation}
\label{optcondeps2m}
A^*Ax+\Lambda^*\frac{\beta p}{\max(\eps^{2-p},|y|^{2-p})}y=A^*b.
\end{equation}
 In order to solve \cref{optcondeps2m}, the following iterative procedure is considered:
\begin{equation}\label{iter2m}
A^*Ax^{k+1}+\Lambda^*\frac{\beta p}{\max(\eps^{2-p},|y^{k}|^{2-p})}y^{k+1}=A^*b,
\end{equation}
where we denote $y^{k}=\Lambda x^{k},$ and the second addends are short for the vectors with components $(\Lambda^*)_{li}\frac{\beta p}{\max(\eps^{2-p},|y_i^{k}|^{2-p})}y_i^{k+1}$. \\
We have the following convergence result. 
\begin{theorem}\label{monotdec2m}
 For $\eps>0$, let $\{x_k\}$ be generated by \ref{iter2m}. Then, $J_{\eps}(x_k)$ is strictly monotonically
	decreasing, unless there exists some $k$ such that $x^k = x^{k+1}$ and $x^k$ satisfies the necessary optimality
	condition \ref{optcondeps2m}. Moreover every  cluster point of $x^k$, of which there exists at least one, is a solution of \ref{optcondeps2m}.
\end{theorem}
\begin{proof}
The proof follows similar arguments to that of Theorem $4.1$, \cite{KI}. Multiplying \ref{iter2m} by $x^{k+1}-x^k$, we get
\begin{eqnarray*}
\frac{1}{2}|A x^{k+1}|^2-\frac{1}{2}|A x^k|^2+\frac{1}{2}|A (x^{k+1}-x^{k})|^2&+&\beta p\left(\frac{1}{\max(\eps^{2-p},|y^{k}|^{2-p})}y^{k+1},y^{k+1}-y^{k}\right)\\ &=&(A^*b,x^{k+1}-x^k).
\end{eqnarray*}
Note that 
\begin{equation}\label{mon1}
\left(\frac{1}{\max(\eps^{2-p},|y^{k}|^{2-p})}y^{k+1},y^{k+1}-y^k\right)=\frac{1}{2}\sum_{i=1}^n\frac{\left(|y_i^{k+1}|^2-|y_i^{k}|^2+|y_i^{k+1}-y_i^{k}|^2\right)}{\max(\eps^{2-p},|y_i^{k}|^{2-p})}
\end{equation}
and
\begin{equation}\label{mon2}
\frac{1}{\max(\eps^{2-p},|y_i^{k}|^{2-p})}\frac{p}{2}(|y_i^{k+1}|^2-|y_i^{k}|^2)=\Psi_\eps'(|y_i^{k}|^2)(|y_i^{k+1}|^2-|y_i^{k}|^2).
\end{equation}
Since  $t \rightarrow \Psi_\eps(t)$ is concave, we have 
\begin{equation}\label{conc}
\Psi_\eps(|y_i^{k+1}|^2)-\Psi_\eps(|y_i^{k}|^2)-\frac{1}{\max(\eps^{2-p},|y_i^{k}|^{2-p})}\frac{p}{2}(|y_i^{k+1}|^2-|y_i^{k}|^2) \leq 0.
\end{equation}
Then, using \ref{mon1}, \ref{mon2}, \ref{conc}, we get
\begin{equation}\label{442m}
J_\eps(x^{k+1}) +\frac{1}{2}|A(x^{k+1}-x^k)|_2^2+\frac{1}{2}\sum_{i=1}^n\frac{\beta p}{\max(\eps^{2-p},|y_i^{k}|^{2-p})}|y_i^{k+1}-y_i^{k}|^2 \leq J_\eps(x^k).
\end{equation}
From \ref{442m}  it follows that $\{x^k\}_{k=1}^\infty$ and thus $\{y^{k}\}_{k=1}^\infty$ are bounded. Then, from \ref{442m}, there exists a constant $\kappa>0$ such that 
\begin{equation}\label{45}
J_\eps(x^{k+1}) +\frac{1}{2}|A(x^{k+1}-x^k)|_2^2+\kappa|y^{k+1}-y^{k}|_2^2  \leq J_\eps(x^k),
\end{equation}
from which we conclude the first  part of the theorem. 
From \ref{45}, we conclude that 
\begin{equation}\label{46}
\sum_{k=0}^\infty |A(x^{k+1}-x^k)|_2^2+|y^{k+1}-y^{k}|_2^2  <\infty.
\end{equation}
Since $\{x^k\}_{k=1}^\infty$ is bounded, there exists a subsequence and $\bar x \in \R^n$ such that $x^{k_l}\rightarrow \bar x$. By \ref{46} and \ref{asslambdi} we have that $x^{k_l+1} \rightarrow \bar x$.
Then, passing to the limit with respect to $k$ in  \ref{iter2m}, we get that $\bar x$ is a solution to \ref{iter2m}.
\end{proof}
In the following proposition we establish the convergence of \ref{optprobeps2m} to \ref{optprob2m} as $\eps$  goes to zero.
\begin{proposition}
Let $\{x_\eps\}_{\eps>0}$ be solution to \ref{optprobeps2m}. Then any cluster point of  $\{x_\eps\}_{\eps>0}$, of which there exists al least one,  is a solution of \ref{optprob2m}.
\end{proposition}
\begin{proof}
From the coercivity of $J_\eps$, we have that $\{x_\eps\}_{\eps}$ is bounded for $\eps$ small and then there exist a subsequence  and $\bar x \in \R^n$ such that $x_{\eps_l} \rightarrow \bar x$. Since $\{x_\eps\}_{\eps}$ solves \ref{optprobeps2m}, by letting  $\eps\to 0$  and using the definition of $\Psi_\eps$, we easily get that  $\bar x$ is a solution of \ref{optprob2m}.
\end{proof}

\section{Monotone algorithm: numerical results}\label{numericsDeps}
The focus of this section is to investigate the performance of the monotone algorithm in practice. For this purpose we choose four problems with matrices $A$ of very different structure: a time-dependent optimal control problem, a fracture mechanics example,  the $M$ matrix  and a stationary optimal control problem.  The latter two problems are studied for the two matrix case. \\
\subsection{The numerical scheme}
For further references it is convenient to recall the algorithm in the following form (see \textbf{Algorithm $1$}).
Note that a  continuation strategy with respect  to the parameter $\eps$ is performed. The initialization and range of $\eps$-values is described for each class of problems below. \\
The algorithm stops when the $\ell^\infty$-norm of the residue of \ref{optcondeps2m} is $O(10^{-3})$ in all the examples, except the fracture problem, where it is $O(10^{-15})$. At this instance, the $\ell^2$-residue is typically much smaller. Thus, we find an approximate solution of the $\eps$-reguralized optimality condition \ref{optcondeps2m}.
The  initialization $x^0$  is chosen in the following way
\begin{equation}\label{initmoneps}
x^0=(A^*A+2\beta \Lambda^*\Lambda)^{-1}A^*b,
\end{equation}
that is, $x^0$ is chosen as the solution of the problem \ref{optprob2m} where the $\ell^p$-term is replaced by the $\ell^2$-norm. 
Our numerical experience shows that for some values of $\beta$ the previous initialization is not suitable, that is, the residue obtained is too big.  In order to get a lower residue, we successfully tested a continuation strategy with respect to increasing $\beta$-values.
\begin{algorithm}[h!]
	\caption{Monotone algorithm + $\eps$-continuation strategy}
	\begin{algorithmic}[1]
		\STATE Initialize $\eps^0$, ‎$x^0$ and set $y^{0}=\Lambda x^0$. Set $k=0$;
		\REPEAT
		\STATE  Solve for $x^{k+1}$ 
		$$
		A^*Ax^{k+1}+\Lambda^*\frac{\beta p}{\max(\eps^{2-p},|y^{k}|^{2-p})}\Lambda x^{k+1}=A^*b.
		$$
		\STATE Set $y^{k+1}=\Lambda x^{k+1}$.
		\STATE Set $k=k+1$.
		\UNTIL{the stopping criterion is fulfilled}.
		\STATE  Reduce $\eps$ and repeat 2.
	\end{algorithmic}
\end{algorithm}

In the presentation of our numerical results, the total number of iterations shown in the tables  takes into account the continuation strategy with respect to $\eps$. However, it does not take into account the continuation with respect to $\beta$.
We remark that in all the experiments presented in the following sections,  the value of the functional for each iterations was checked to be monotonically decreasing accordingly to Theorem \ref{monotdec2m}.\\
The following notation will hold for the rest of the paper.  For $x \in \R^n$ we will denote  $|x|_0=\#\{i\, :\, |x_i|> 10^{-10}\},$  $|x|_0^c=\#\{i\, :\, |x_i|\leq 10^{-10}\},$ and by $|x|_2$  the euclidean norm of $x$. 

 \subsection{Time-dependent control problem}\label{controlone}
We consider the linear control system
$$
\frac{d}{dt} y(t)=\mathcal{A} y(t)+B u(t), \quad y(0)=0,
$$
that is,
\begin{equation}\label{LCSfinalstate}
y(T)=\int_0^T e^{\mathcal{A}(T-s)} B u(s) ds,
\end{equation}
where the linear closed operator $\mathcal{A}$ generates a $C_0$-semigroup $e^{\mathcal{A}t}$, $t\geq 0$ on the state space $X$. More specifically, we consider the  one dimensional controlled heat equation for $y=y(t,x)$:
\begin{equation}\label{actionscontrol}
y_t=y_{xx}+b_1(x)u_1(t)+b_2(x)u_2(t), \quad x \in (0,1),
\end{equation}
with homogeneous boundary conditions $y(t,0)=y(t,1)=0$ and thus $X=L^2(0,1)$. The differential operator $\mathcal{A}y=y_{xx}$ is discretized in space by the second order finite difference approximation with $n=49$ interior spatial nodes ($\Delta x=\frac{1}{50}$). We use two time dependent controls $\overrightarrow u=(u_1,u_2)$ with corresponding spatial control distributions $b_i$ chosen as step functions:
$$
b_1(x)=\chi_{(.2,.3)}, \quad b_2(x)=\chi_{(.6,.7)}.
$$
The control problem consists in finding the control function $\overrightarrow u$ that steers the state $y(0)=0$ to a neighborhood of the desired state $y_d$ at the terminal time $T=1$. We discretize the problem in time by the mid-point rule, i.e.
\begin{equation}\label{midpoint}
A \overrightarrow  u=\sum_{k=1}^m e^{\mathcal{A}\left(T-t_{k}-\frac{\Delta t}{2}\right)} (B \overrightarrow u)_k \Delta t,
\end{equation}
where $\overrightarrow u=(u_1^1,\cdots, u_1^m,u_2^1,\cdots u_2^m)$ is a discretized control vector whose coordinates represent the values at the mid-point of the intervals $(t_k,t_{k+1})$. Note that in \ref{midpoint} we denote by $B$ a suitable rearrangement of the matrix $B$ in \ref{LCSfinalstate} with some abuse of notation. A uniform step-size $\Delta t=\frac{1}{50}$ ($m=50$) is utilized. The solution of the control problem is based on the sparsity formulation \ref{optprob2m}, where $\Lambda$ is the backward difference operator acting independently on each component of the control, that is, $\Lambda=m(I_2\otimes D)$ where $I_2$ is the $2\times2$ identity matrix and $D\, :\, \R^m \rightarrow \R^m$ is as follows
\begin{equation}\label{Dcontrol}
D=\left(\begin{array}{ccccc}
1& 0& 0& \cdots& 0\\ -1& 1& 0& \cdots& 0\\\vspace{0.2cm}\\ 0& \cdots& 0&-1&1\
\end{array}\right).
\end{equation}
 Also,
 $b$ in \ref{optprob2m} is the discretized target function chosen as the Gaussian distribution $y_d(x)=0.4\,\mbox{exp}(-70(x-.7)^2))$ centered at $x=.7$.
 That is, we apply our algorithm for the discretized optimal control problem in time and space where $x$ from \ref{optprob2m} is the discretized control vector $u \in \R^{2m}$ which is mapped by $A$ to the discretized output $y$ at time $1$ by means of \ref{midpoint}. Moreover $b$ from \ref{optprob2m} is the discretized state $y_d$ with respect to the spatial grid $\Delta x$. The parameter $\eps$ was initialized with $10^{-3}$ and decreased down to $10^{-8}$.
Note that, since the second control distribution is well within the support of the desired state $y_d$ we expect the authority of this control to be stronger than that of the first one, which is away from the target. \\
In Table $1$  we report the results of our tests for $p=.5$ for $\beta$ incrementally increasing by factor of $10$ from $10^{-3}$ to $1$. We report only the values for the second control $u_2$ since the first control $u_1$ is always zero. In the third row we see that $(|Du_2|_0)^c$ increases with $\beta$, consistent with our expectation.
  Note also that  the quantity $|Du_2|_p^p$ decreases for $\beta$ increasing. \\
For any $i=1,\cdots, m$, we say that $i$ is a  singular component of the vector $D u_2$ if $i \in\{i \, :\, |(Du_2)_i|<\eps\}$. In particular, note that the singular components are the ones where the $\eps$-regularization is most influential. In the sixth row of Table $1$ we show their number at the end of the $\eps$-path following scheme (denoted by $Sp$) and we observe  that it concides with the quantity $|Du_2|_0^c$, which is reassuring the validity of our $\eps$-strategy. \\
The algorithm was also tested for values of $p$ near to $1$, e.g. for $p=.9$. The results obtained shows a less piecewise constant behaviour of the solution with respect to the ones for $p=.5$.
Finally, we remark that if we change the initialization \ref{initmoneps}, the method converges to the same solution with no remarkable modifications in the number of iterations.
\begin{table}[tbhp]\caption{Sparsity in a time-dependent control problem, $ p=.5$, mesh size $h=\frac{1}{50}$. Results obtained by \textbf{Algorithm $1$}.}
	\centering
	\begin{tabular}{|l|c|c|c|c|}
		\hline
		{\bf $\beta$ }  &$10^{-3}$&$10^{-2}$&$10^{-1}$&$1$\\
		\hline
		no. of iterates &630&635&29&19\\
		\hline
		{\bf $|Du_2|_0^c$ } & 97&99&100&100 \\
		\hline
		{\bf $|Du_2|^p_p$ }& 158&16.7&$6*10^{-5}$&$10^{-4}$ \\
		\hline
		$\mbox{ Residue }$ &  $3*10^{-3}$ &$2*10^{-3}$ &$1.2*10^{-3}$&$2.5*10^{-10}$\\
		\hline
		$\mbox{Sp}$ &97&99&100&100\\
		\hline
	\end{tabular}
\end{table}

\subsection{Quasi-static evolution of cohesive fracture models}\label{fract}
In this section we focus on a modelling problem for  quasi-static evolutions of cohesive fractures. This kind of problems require the minimization of an energy functional, which has two components: the elastic energy and the cohesive fracture energy. The underlying idea is that the fracture energy is released gradually with the growth of the crack opening.  The cohesive energy, denoted by $\theta$, is assumed to be a monotonic non-decreasing  function of the jump amplitude of the displacement,  denoted by $\llbracket u \rrbracket$. Cohesive energies were introduced independently by Dugdale  \cite{DG} and Barenblatt  \cite{BA}, we refer to \cite{PI13} for more details on the models.  Let us just remark that the two models differ mainly in the evolution of  the derivative $\theta'(\llbracket u\rrbracket)$, that is, the \textit{bridging force}, across a crack amplitude $\llbracket u \rrbracket$. In Dugdale's model this force keeps a constant value up to a critical value of the crack opening and then drops to zero. In Barenblatt's model, the dependence of the force on $\llbracket u \rrbracket$ is continuous and decreasing. \\
In this section we test the $\ell^p$-term $0<p<1$ as a  model for the cohesive energy. In particular,  the cohesive energy is not differentiable in zero and the bridging force goes to infinity when the jump amplitude goes to zero. Note also that the bridging force goes to zero when the jump amplitude goes to infinity.\\
Let us  introduce all the elements that we need for the rest of the section. We consider the one-dimensional domain $\Omega=[0,2l]$ with $l>0$ and we denote by $u\,:\, \Omega \rightarrow \R$ the displacement function. The deformation of the domain is given by an external force which we express in terms of an external displacement function $g\,:\,\Omega\times [0,T] \rightarrow \R$. We require that the displacement $u$ coincides with the external deformation, that is
$$
u|_{\partial \Omega}=g|_{\partial \Omega}.
$$
We denote by $\Gamma$ the point of the (potential) crack, which we chose as the midpoint $\Gamma=l$ and  by $\theta(\llbracket u \rrbracket)_\Gamma$ the value of the cohesive energy  $\theta$ on the crack amplitude of the displacement $\llbracket u \rrbracket$ on $\Gamma$. Since we are in a quasi-static setting, we introduce the time discretization $0=t_0<t_1< \cdots <t_T=T$ and look for the equilibrium configurations which are minimizers of the energy of the system. This means that  for each $i \in \{0, \cdots, T\}$ we need to minimize the energy of the system
$$
J(u)=\frac{1}{2}\int_{\Omega\backslash \Gamma}|\nabla u|^2 dx +\beta \theta(\llbracket u \rrbracket)_\Gamma
$$ 
with respect to a given boundary datum $g$:
$$
u^*\in \argmin_{u=g(t_i) \mbox{ on } \partial \Omega} J(u),
$$
where $\beta>0$ in $J(u)$ is a  material parameter.
In particular, we consider the following type of cohesive energy
$$
\theta(\llbracket u \rrbracket)=|\llbracket u \rrbracket|^p,
$$
for $p \in (0,1)$.
 We divide $\Omega$ into $2N$ intervals  and  approximate the displacement function with a function $u_h$ that is piecewise linear on $\Omega\backslash \Gamma$ and has two degrees of freedom on $\Gamma$ to represent correctly the two lips of the fracture, denoting with $u_{N}^{-}$ the degree on $[0,l]$ and $u_{N}^{+}$ the one on $[l,1]$. We discretize the problem in the following way
\begin{equation}\label{fractdiscr}
J_h(u_h)=\frac{1}{2} \sum_{i=1}^{2N} \frac{N}{l} |u_i -u_{i-1}|^2+\beta |\llbracket u_N \rrbracket|^p,
\end{equation}
where if $i\leq N$ we identify $u_N=u_N^-$ while for $i>N, u_N=u_N^+$. We remark that the jump of the displacement is not taken into account in the sum, and the gradient of $u$ is approximated with finite difference of first order.
 The Dirichlet condition is applied on $\partial \Omega=\{0,2l\}$ and the external displacememt is chosen as
 $$
 u(0,t)=0, \quad u(2l,t)=2lt.
 $$
 To enforce the boundary condition in the minimization process, we add it to the energy functional as a penalization term. Hence, we solve the following  unconstrained minimization problem
 \begin{equation}\label{ff1}
 \min \frac{N}{2l}|A u_h -g|_2^2+\beta | \llbracket u_N \rrbracket|^p,
 \end{equation}
 where the operator $A \in \R^{(2N+1)\times (2N+1)}$ is given by
 $$
A=\left[\begin{array}{c}
\bar D\\
0\, \,\cdots\,\, 0 \, \,\gamma\end{array} \right].
$$
Here $\bar D \in \R^{2N\times (2N+1)}$  denotes the backward finite difference operator  $D	\, : \,\R^{2N+1} \to \R^{2N+1}$
 without the $N+1$ row, where $D$ is defined in \ref{Dcontrol}.
  Moreover
$g \in \R^{2N+1}$ in \ref{ff1} is given by $g=(0,\cdots, \gamma 2lt_i)'$ and $\gamma$ is the penalization parameter. To compute the jump between the two lips of the fracture, we introduce the operator $D_f:\R^{2N+1} \to \R$ defined as $D_f=(0,\cdots, -1,1,0,\cdots,0)$  where $-1$ and $1$ are respectively in the $N$ and $N+1$ positions.
Then we write the functional \ref{ff1} as follows
 \begin{equation}\label{ff2}
 \min \frac{N}{2l}|A u_h -g|_2^2+\beta | D_fu|^p,
 \end{equation}
Note that    $\mbox{KerA }=0$, hence assumption \ref{asslambdi} is satisfied and existence  of a minimizer for \ref{ff2} is  guaranteed.\\
Our numerical experiments were conducted with a discretization in $2N$ intervals with $N=100$ and a prescribed potential crack $\Gamma=0.5$. The time step in the time discretization of $[0,T]$ with $T=3$ is set to $dt=0.01$. The parameters of the energy functional $J_h(u_h)$ are set to $\beta=1, \gamma=50$. The parameter $\eps$  is  decreased from $10^{-1}$  to $10^{-12}$. \\
In  Figures $1$ we report three time frames to represent the evolutions of the crack obtained with  \textbf{Algorithm $1$} for two different values of $p$, that is, $p=.01, .1$ respectively. Each time frame consists of  three different time steps $(t_1, t_2, t_3)$, where $t_2, t_3$ are chosen as the first instant where the prefacture and the fracture appear.
The evolution presents the three phases that we expect from a cohesive fracture model:
\begin{itemize}
\item \textit{Pure elastic deformation}: in this case the jump amplitude is zero and the gradient of the displacement is constant in $\Omega \backslash \Gamma$;
\item \textit{Prefracture}: the two lips of the fracture do not touch each other, but they are not free to move. The elastic energy is still present.
\item \textit{Fracture}: the two parts are free to move. In this final phase the gradient of the displacement (and then the elastic energy) is zero.
\end{itemize}
Moreover  we remark that the  formation  of the crack is anticipated for smaller values of $p$. As we see in Figure $1$, for $p=.01$  prefracture and fracture  are reached at $t=.3$ and $t=1.5$ respectively. As $p$ is increased to $p=.1$,  prefracture and fracture occur at $t=1$ and $t=3$ respectively.
Finally we remark that in our  experiments the residue is  $O(10^{-16})$ and the number of iterations is small, e.g. $12, 15$ for $p=.01, .1$ respectively.

\begin{figure}[h!]
\centering
\subfloat[$t=0.2$]{\includegraphics[height=5.2cm, width=3cm]{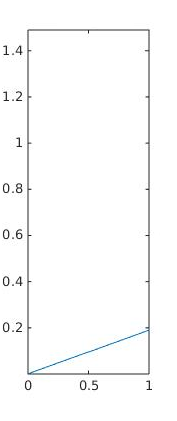}}
\subfloat[$t=0.3$]{\includegraphics[height=5.2cm, width=3cm]{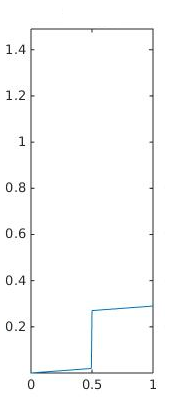}}
\subfloat[$t=1.5$]{\includegraphics[height=5.2cm, width=3cm]{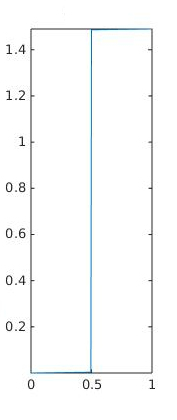}}
%
\hspace{0.3cm}
\subfloat[$t=0.9$]
{
\includegraphics[height=5.2cm, width=3cm]{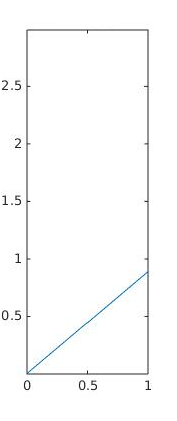}
}
\subfloat[$t=1$]
{
\includegraphics[height=5.2cm, width=3cm]{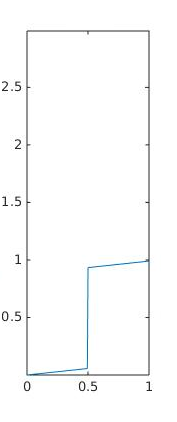}
}
\subfloat[$t=3$]
{
\includegraphics[height=5.2cm, width=3cm]{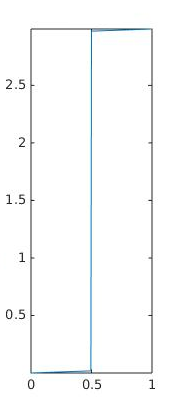}
}
\caption{Three time-step evolution of the displacement for  $p=.01$, $t = .2, .3, 1.5$  (up), $p=.1$, $t=.9, 1, 3$ (down). Results obtained by \textbf{Algorithm $1$}.}
\end{figure}

\subsection{M-matrix}\label{2MM}
We consider
\begin{equation}
\label{optprobM2M}
\min_{x \in \R^{n^2}}\frac{1}{2}|A x-b|_2^2+\beta |\Lambda x|^p_p,
\end{equation}
where $A$ is the backward finite difference gradient
\begin{equation}\label{A}
A=(n+1)\left(\begin{array}{c} G_1\\G_2\end{array}\right),
\end{equation}
with $G_1 \in \R^{n(n+1)\times n^2}, G_2 \in \R^{n(n+1)\times n^2}$ given by
$$
G_1=I \otimes D, \quad G_2=D \otimes I.
$$
Here $I$ is the $n\times n$ identity matrix, $\otimes$ denotes the tensor product, and  $D\in \R^{(n+1)\times n}$ is given by
\begin{equation}\label{D}
D=\left(\begin{array}{ccccc}
1& 0& 0& \cdots& 0\\ -1& 1& 0& \cdots& 0\\ \vspace{0.2cm}\\ 0& \cdots& 0&-1&1\\0&\cdots&0&0&-1 
\end{array}\right).
\end{equation}
Then $A^* A$ is an $M$ matrix coinciding with the $5$-point star discretization on a uniform mesh on a square of the Laplacian with Dirichlet boundary conditions.  Note that \ref{optprobM2M} can be equivalently expressed as
\begin{equation}
\label{optprob2}
\min_{x \in \R^{n\times n}}\frac{1}{2}|A x|_2^2-(x,f)+\beta |\Lambda x|^p_p,
\end{equation}
where $f=A^* b$. If $\beta=0$ this is the discretized variational form of the elliptic equation
\begin{equation}
\label{elleq}
-\Delta y=f \mbox{ in } \Omega, \quad y=0 \mbox{ on } \partial \Omega.
\end{equation}
For $\beta>0$ the variational problem \ref{optprob2} gives a solution piecewise constant enhancing behaviour.\\
Our tests were conducted with $f$  chosen as discretization of $f=10 x_1\mbox{sin}(5x_2) \mbox{cos}(7 x_1)$ and 
$$
\Lambda=(n+1)\left(\begin{array}{c} D_1\\D_2\end{array}\right),
$$
where $D_1 \in \R^{n^2\times n^2}, D_2 \in \R^{n^2\times n^2}$ are defined as follows
\begin{equation}\label{d1d2}
D_1=I\otimes D, \quad  D_2=D\otimes I,
\end{equation}
 and $D \in \R^{n\times n}$ is  the backward difference operator defined in \ref{D} without the $n+1$-row.
 The parameter $\eps$ was initialized with $10^{-1}$ and decreased  to $10^{-6}$.\\
In Tables $2$ we show the performance of \textbf{Algorithm $1$} for  $p=.1$,  $h=1/64$ as mesh size and  $\beta$ incrementally increasing by factor of $10$ from $10^{-4}$ to $10$. In Figure $2$ we report the graphics of the solutions for different values of $\beta$ between $.01$ and $.3$ where most changes occur in the graphics.\\
We observe significant differences in the results with respect to different values of $\beta$. Consistently with our expectations, $|\Lambda x|^c_0$ increases with $\beta$ (see the third row of Table $2$). For example, for $\beta=1, 10$, we have $|\Lambda x|^c_0=7938$, or equivalently, $|\Lambda x|_0=0$, that is, the solution to \ref{optprob2} is  constant. Moreover the fourth row shows that  $|\Lambda x|^p_p$  decreases when $\beta$ increases.\\
The fifth row exhibits the $\ell^\infty$ norm of the residue, which is $O(10^{-4})$ for  all the considered $\beta$. 
We remark that  the number of iterations is sensitive with respect to $\beta$, in particular it increases when $\beta$ is increasing from $10^{-4}$ to $10^{-1}$ and then it decreases significantly for $\beta=1,10$.\\
The algorithm  was also tested for different values of $p$. The results obtained show dependence on $p$, in particular $|\Lambda x|^c_0$  decreases  as $p$ is increasing.  For example, for $p=.5$ and $\beta=.1$ we have $|\Lambda x|^c_0=188, |\Lambda x|^p_p=528$.\\
In the sixth row of Table $2$ we show the number of singular components of the vector $\Lambda x$ at the end of the $\eps$-path following scheme, that is, $Sp:=\#\{ i\,\, |\,\, |(\Lambda x)_i|<\eps\}$. For most values of $\beta$, we note that  $Sp$ is comparable to  $|\Lambda x|_0^c$. This again confirms that the $\eps$-strategy is effective.\\
Finally, we remark that if we modify the initialization \eqref{initmoneps}, the method converges to the same solution with no remarkable modifications in the number of iterations, which is a sign for the global nature of the algorithm.

\begin{table}[tbhp]
	\caption{$M$-matrix example, $\Lambda=(n+1)[D_1;D_2],  p=.1$, mesh size $h=\frac{1}{64}$. Results obtained by \textbf{Algorithm $1$}.} 
	\centering
	\begin{tabular}{|l|c|c|c|c|c|c|}
		\hline\noalign{\smallskip}
		{\bf $\beta$ }   &$10^{-4}$& $10^{-3}$ &$10^{-2}$&$10^{-1}$ & $1$ \\		\hline
		no. of iterates  & 1701&2469& 3929&4254& 14   \\
		\noalign{\smallskip}\hline\noalign{\smallskip}
		{\bf $|\Lambda x|^c_0$ } &16 &103& 791&5384& 7938\\
		{\bf $|\Lambda x|^p_p$ }& $6*10^3$&$5.8*10^3$&$ 5* 10^{3}$&$2.4*10^3$&584\\
		$\mbox{ Residue }$ & $2.7*10^{-7}$ & $5.5*10^{-6}$ & $ 9*10^{-5}$ &  $9*10^{-4}$  & $3*10^{-12}$\\
		$Sp$  & 247&696&2097& 5599& 7938\\
		\noalign{\smallskip}\hline
	\end{tabular}
\end{table}

\begin{figure}[h!]
\centering
\subfloat[$\beta=0.01$]
{
\includegraphics[height=4cm, width=3.6cm]{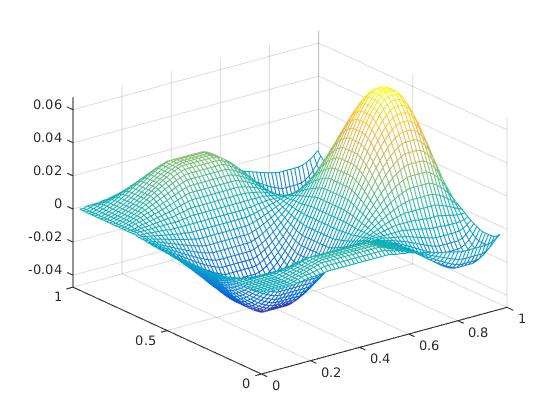}
}
\subfloat[$\beta=0.05$]
{
\includegraphics[height=4cm, width=3.6cm]{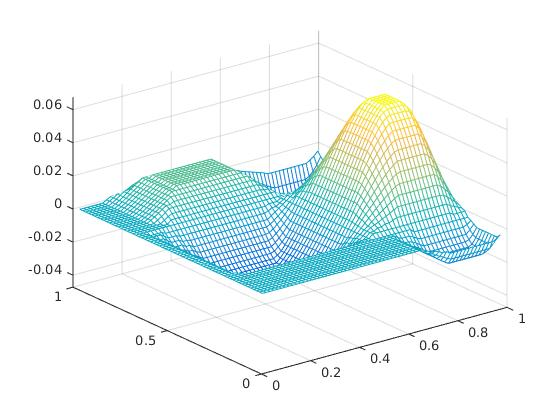}
}
\subfloat[$\beta=0.08$]
{
\includegraphics[height=4cm, width=3.6cm]{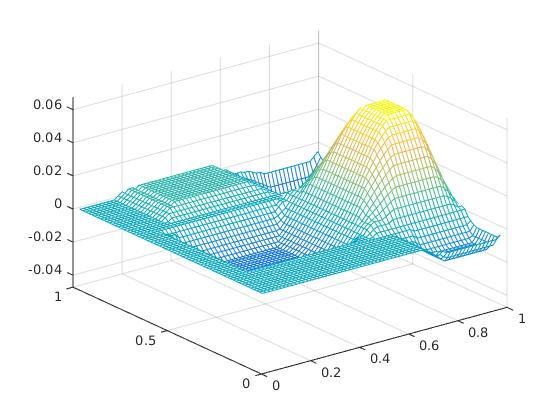}
}
\newline
\subfloat[$\beta=0.12$]
{
\includegraphics[height=4cm, width=3.6cm]{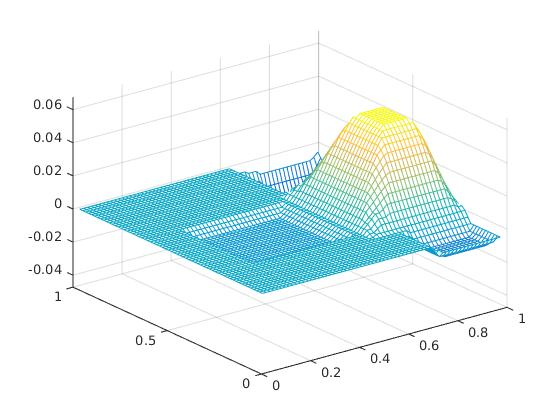}
}
\subfloat[$\beta=0.15$]
{
\includegraphics[height=4cm, width=3.6cm]{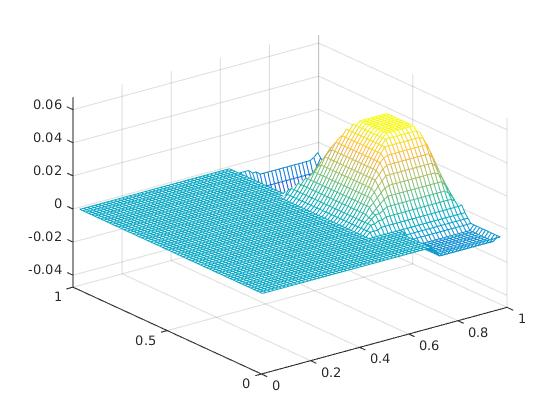}
}
\subfloat[$\beta=0.3$]
{
\includegraphics[height=4cm, width=3.6cm]{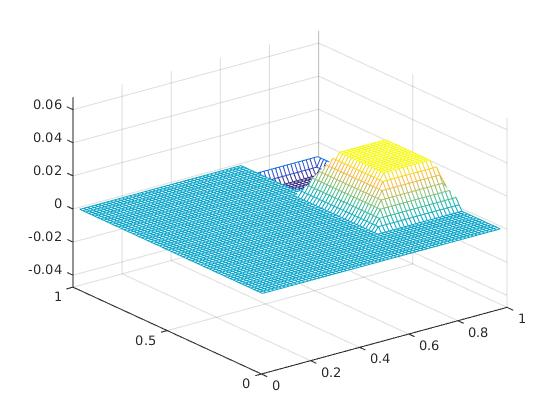}
}
\caption{Solution of the M-matrix problem, $p=.1, \Lambda=(n+1)[D_1;D_2]$, mesh size $h=\frac{1}{64}$. Results obtained by \textbf{Algorithm $1$}.}
\end{figure}

\vspace{1cm}

\begin{remark}\label{Mmatrixrem}
The algorithm was also tested in the following two particular cases: $\Lambda=I$, where $I$ is the identity matrix of size $n^2$, and  $\Lambda=(n+1)D_1$, where $D_1$ is as in \cref{d1d2}.  
\\
In the case $\Lambda=I$ the variational problem \cref{optprob2} for $\beta>0$  gives a sparsity enhancing solution for the elliptic equation \cref{elleq}, that is, the displacement $y$ will be $0$ when the forcing $f$ is small. Indeed, in this case we have sparsity of the solution increasing with $\beta$. Also, the residue is $O(10^{-8})$ and the number of iterations is considerably smaller than in the two matrix case. \\
For the case $ \Lambda=(n+1)D_1$ we show the graphics in Figure $3$. Comparing the graphs for $\beta=.3$ in Figure $2$ and Figure $3$ we can find subdomains where the solution is only unidirectionally piecewise constant in Figure $3$ and piecewise constant in Figure $2$. The number of iterations, $|\Lambda x|_0^c, |\Lambda x|_p^p$ and the residue are comparable to the ones of Table $2$.
\begin{figure}[H]
\centering
\subfloat[$\beta=0.01$]
{
\includegraphics[height=4cm, width=3.6cm]{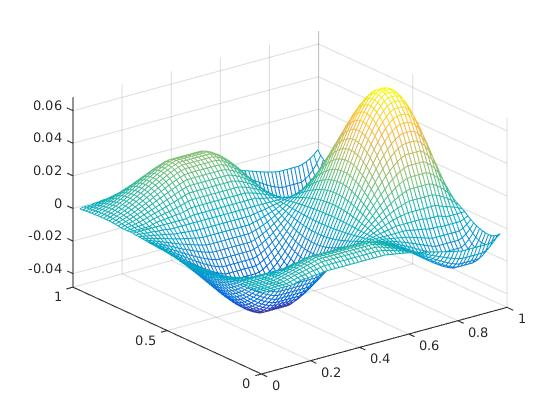}
}
\subfloat[$\beta=0.1$]
{
\includegraphics[height=4cm, width=3.6cm]{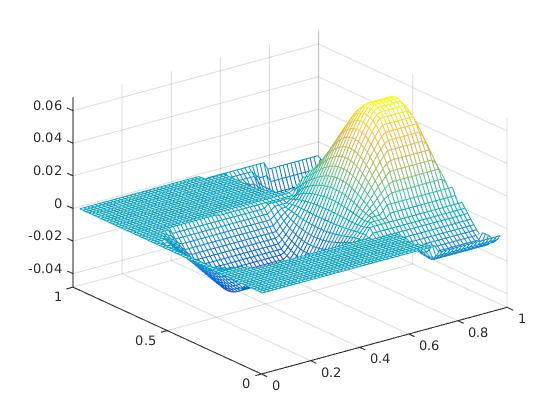}
}\subfloat[$\beta=0.3$]
{
\includegraphics[height=4cm, width=3.6cm]{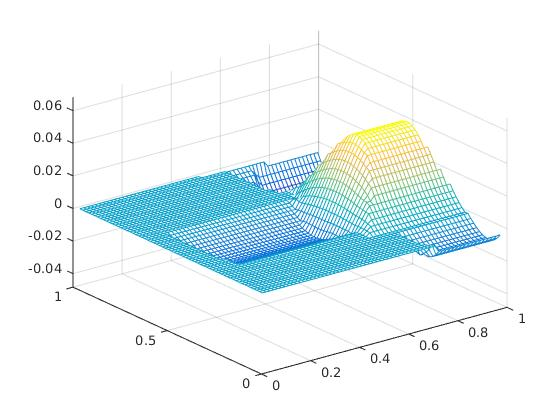}
}
\caption{Solution of the M-matrix problem, $p=.1, \Lambda=(n+1)D_1$, mesh size $h=\frac{1}{64}$. Results obtained by \textbf{Algorithm $1$}.}
\end{figure}
\end{remark}

\subsection{Elliptic control problem}\label{2MC}
We consider the following two dimensional control problem 
\begin{equation}
\label{optprobC2M}
\inf \frac{1}{2}|y-y_d|_2^2+\beta |\nabla u|^p_p, \quad p \in (0,1],
\end{equation}
where we minimize over $u \in L^p(\Omega)$ such that $\nabla u \in L^p(\Omega)$, $\Omega$ is the unit square, $y_d \in L^2(\Omega)$ is a given target function, and  $y \in L^2(\Omega)$ satisfies
\begin{equation}\label{eq2MControl}
\left\{
\begin{array}{lll}
-\Delta y=u \quad &\mbox{ in }\,\, \Omega\\
y=0 \quad & \mbox{ in }\,\,\partial \Omega.
\end{array}
\right.\,
\end{equation}
We discretize   \cref{optprobC2M} by the following $\frac{1}{n}$-mesh size discretized minimization problem
\begin{equation}
\label{optprob2MControl}
\min_{u \in \R^{n^2}}\frac{1}{2}|Eu-b|_2^2+\beta |\Lambda u|^p_p,
\end{equation}
where $E=(A^* A)^{-1}$,  $A$ is as in \cref{A}  (that is, $A^*A$ is the $5$-point star discretization on a uniform mesh on a square of the Laplacian with Dirichlet boundary condition), $
\Lambda=(n+1)\left(\begin{array}{c} D_1\\D_2\end{array}\right)
$ is as in  \cref{2MM} and $b$ is the discretized  target function.\\
For numerical reasons, in order to avoid the inversion of the matrix $A^*A$ we multiply  the necessary optimality condition \cref{optcondeps2m}  by $(E^{-1})^*$ and we get
\begin{equation}\label{optprob2MControl2}
Eu+(E^{-1})^*\Lambda^*\frac{\beta p}{\max(\eps^{2-p},|y|^{2-p})}y^1=b,
\end{equation}
where $y=\Lambda u$. We introduce
$$
z=Eu, \quad p=(\Lambda^*N\Lambda)u,
$$ 
where  we denote by $N$ the diagonal matrix with $i$-entry $(N)_{ii}=\frac{\beta p}{\max(\eps^{2-p},|y_i|^{2-p})}, \, i=1,\cdots n^2$. Since $E^{-1}=A^*A$, we can express \cref{optprob2MControl2} in the  form
\begin{equation}\label{system2C2}
\left\{
\begin{array}{lll}
A^* Az=u \\
A^*A p=b-z\\
(\Lambda^*N\Lambda)u=p.
\end{array}
\right.\,
\end{equation}
To solve \cref{system2C2} the following iteration procedure  is used
\begin{equation}\label{systemmatr}
\left(\begin{array}{ccc}
I& 0& A^*A\\ 0& \Lambda^*N^{k}\Lambda& -I\\ A^*A& -I& 0
\end{array}\right)
\left(\begin{array}{c}
z^{k+1}\\ u^{k+1}\\ p^{k+1}
\end{array}\right)
=
\left(\begin{array}{c}
b\\ 0\\ 0
\end{array}\right)
\end{equation}
where   we denote by  $N^{k}$ the diagonal matrix with $i$-entry $(N^{k})_{ii}=\frac{\beta p}{\max(\eps^{2-p},|y_i^{k}|^{2-p})}$ for $i=1,\cdots, n^2$ and $y^{k}=\Lambda u^k$. Note that the system  matrix \cref{systemmatr} is symmetric.\\
In our  tests the target $b$ is chosen as the image through $E$ of the linear interpolation inside $[.2,.8]\times[.2,.8]\setminus [.3,.7]\times [.3,.7]$ of the step function $1000 \chi_{[.3,.7]\times [.3,.7]}$.
  The parameter $\eps$ was initialized with $10^{-1}$ and decreased  to $10^{-6}$.\\
In Table $3$ we report the results of our test for  $h=\frac{1}{64}$, $p=.1$ and $\beta$ incrementally increasing by factor of $10$ from $10^{-3}$ to $1$. 
 As expected, when $\beta$ increases,  $|\Lambda u|_0^c$  increases and $|\Lambda u|^p_p$ decreases. In Figure $5$ we show the graphics of the solution for different values of $\beta$, thus showing the enhancing piecewise constant behaviour of the solution. 
\begin{table}[tbhp]
	\caption{Sparsity in an elliptic control problem,  $p=.1$, mesh size $h=\frac{1}{64}$.Results obtained by \textbf{Algorithm $1$}.}
	\centering
	\begin{tabular}{|l|c|c|c|c|c|c|c|c|}
		\hline
		{\bf $\beta$ }  &$10^{-3}$& $10^{-2}$&$10^{-1}$& $1$\\
		\hline
		no. of iterates &102&119&5204&10440\\
		\hline
		{\bf $|\Lambda u|_0^c$ } & 799&1486&1673& 2376 \\	
		\hline
		{\bf $|\Lambda u|^p_p$ }& $3.2*10^{4}$& $2.6*10^4$&$2.6*10^4$& $1.2*10^4$ \\
		\hline
		$\mbox{ Residue }$ &  $1.6*10^{-5}$ & $ 2.4*10^{-4}$&$2*10^{-3}$ & $7*10^{-3}$\\
		\hline
	\end{tabular}
\end{table}

\begin{figure}[h!]
\centering
\subfloat[$\beta=0.01$]
{
\includegraphics[height=3.8cm, width=3.8cm]{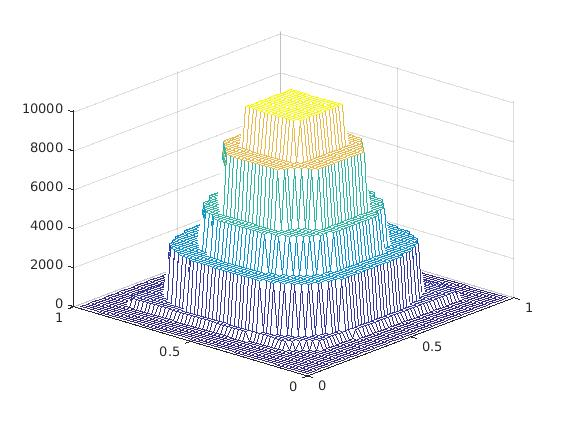}
}
\subfloat[$\beta=0.1$]
{
\includegraphics[height=3.8cm, width=3.8cm]{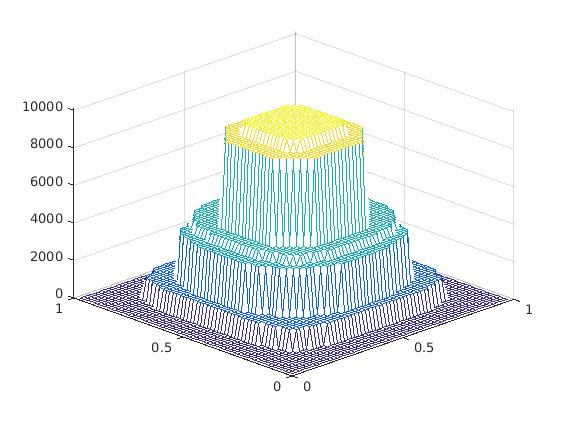}
}
\subfloat[$\beta=1$]
{
\includegraphics[height=3.8cm, width=3.8cm]{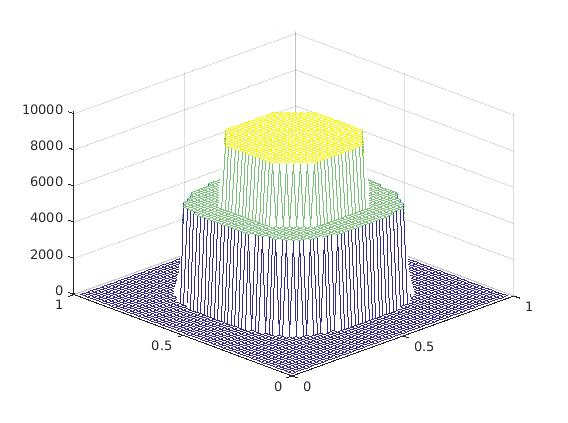}
}
\caption{Solution of the elliptic control problem, $p=.1$, mesh size $h=\frac{1}{64}$. Results obtained by \textbf{Algorithm $1$}.}
\end{figure}
From our  tests we conclude that the monotone algorithm is reliable to find a solution of the $\eps$-regularized optimality condition \cref{optcondeps2m} for a diverse spectrum of problems. It is also stable with respect to the choice of initial conditions. According to the last rows of Tables $1,2,3$ we have that  $\#\{i\, |\, |(\Lambda x)_i|\leq 10^{-10}\}$ is typically very close to the number of singular components at the end of the $\eps$-path following scheme. Depending on the choice of $\beta$ the algorithm requires on the order of $O(10^2)$ to $O(10^3)$ iterations to reach convergence. In the following sections we aim at analysing an alternative algorithm for which the iteration number is smaller, despite  the fact that the convergence can be proved only in special cases.

\section{The active set monotone algorithm for the optimality conditions}\label{alg}
In the following we discuss an algorithm which aims at finding a solution of the original unregularized problem
\begin{equation}
\label{foptprob}
\min_{x \in \R^n}J(x)=\frac{1}{2}|A x-b|_2^2+\beta |\Lambda x|^p_p,
\end{equation}
where  $A \in \mathbb{M}^{m\times n},  b \in \R^m, p \in (0,1]$ and $ \beta \in \R^+$ are as in \cref{sec:optcond} and $\Lambda \in \mathbb{M}^{n\times n}$ is a regular matrix. Existence for the problem \cref{foptprob} follows from  \cref{2mex}.\\
 First  the necessary optimality conditions for problem \cref{foptprob} in the form of a complementary systems are derived. Then an active-set strategy is proposed relying on the form of the optimality condition. 
Convergence of the primal-dual active set strategy is proven in the case $\Lambda=I$. Finally, the results of our numerical tests in two different situations are reported  in  \cref{numericsD}.
\subsection{Necessary optimality conditions}
 For any matrix  $A \in \mathbb{M}^{m\times n}$, we denote  by $A_i$ the $i$-th column of $A$. We have the following necessary optimality conditions. 
 \begin{theorem}\label{thmnecopt} 
 Let $\bar x$ be a global minimizer of \cref{foptprob} and denote $\bar y=\Lambda \bar x$. Then
\begin{equation}\label{optcond}
\left\{
\begin{array}{lll}
A^*(Ax-b)+\Lambda^*\lambda=0\\
(\Lambda \bar x)_i=0 \quad &\mbox{ if }\,\, \left||\tilde{A}_i|^2 \bar y_i +\lambda_i\right|<\mu_i\\
\noalign{\smallskip}
|(\Lambda x)_i|>0 \mbox{ and } \lambda_i=\frac{\beta p (\Lambda\bar x)_i}{|(\Lambda\bar x)_i|^{2-p}} \quad & \mbox{ if }\,\, \left||\tilde{A}_i|^2 \bar y_i +\lambda_i\right|>\mu_i,
\end{array}
\right.\,
\end{equation}
where $\tilde{A}=A \Lambda^{-1}$, $\mu_i=\beta^{\frac{1}{2-p}}(2-p)(2(1-p))^{-\frac{1-p}{2-p}}|\tilde{A}_i|_2^{1-\frac{p}{2-p}}$. If $\left||\tilde{A}_i|^2\bar y_i+\lambda_i\right|=\mu_i$, then $(\Lambda\bar x)_i=0$ or $(\Lambda\bar x)_i=\left(\frac{2\beta (1-p)}{|\tilde{A}_i|_2^2}\right)^{\frac{1}{2-p}}\mbox{ sgn } (|\tilde{A}_i|^2\bar y_i+\lambda_i)$.
\end{theorem}
\begin{proof}
Note that if $\bar{x}$ is a global minimizer of \cref{foptprob}, then $\bar{y}=\Lambda \bar x$ is a global minimizer of 
\begin{equation}\label{optprobid}
\min_{y \in \R^n}\frac{1}{2}|\tilde{A}y-b|_2^2+\beta |y|^p_p,
\end{equation}
where $\tilde{A}=A \Lambda^{-1}$.  Then, by the same arguments as in \cite{KI}, Theorem $2.2$ applied to the functional \cref{optprobid},  we get the following property of global minimizers
\begin{equation}\label{optcondproof}
\left\{
\begin{array}{lll}
\bar y_i=0 \quad &\mbox{ if }\,\, |(\tilde{A}_i, f_i)|<\mu_i\\
|y_i|>0 \mbox{ and } (\tilde{A}_i,\tilde{A}\bar y-b)+\frac{\beta p \bar y_i}{|\bar y_i|^{2-p}}=0 \quad & \mbox{ if }\,\, |(\tilde{A}_i, f_i)|>\mu_i,
\end{array}
\right.\,
\end{equation}
where $f_i=b-\tilde{A} y+\tilde{A}_i\bar y_i$ and $\mu_i=\beta^{\frac{1}{2-p}}(2-p)(2(1-p))^{-\frac{1-p}{2-p}}|\tilde{A}_i|_2^{1-\frac{p}{2-p}}$. Moreover, if $|(\tilde{A}_i,f_i)|=\mu_i$, then $\bar y_i=0$ or $\bar y_i=\left(\frac{2\beta (1-p)}{|\tilde{A}_i|_2^2}\right)^{\frac{1}{2-p}}\mbox{ sgn } ((\tilde{A}_i,f_i))$. We introduce the multiplier $\lambda$ and we write 	\cref{optcondproof} in the following way
\begin{equation}\label{optcondproof2}
\left\{
\begin{array}{lll}
\tilde{A}^*(\tilde{A}y-b)+\lambda=0\\
\bar y_i=0 \quad &\mbox{ if }\,\, \left||\tilde{A}_i|^2 \bar y_i +\lambda_i\right|<\mu_i\\
\noalign{\smallskip}
|y_i|>0 \mbox{ and } \lambda_i=\frac{\beta p \bar y_i}{|\bar y_i|^{2-p}} \quad & \mbox{ if }\,\, \left||\tilde{A}_i|^2 \bar y_i +\lambda_i\right|>\mu_i.
\end{array}
\right.\,
\end{equation}
Then the optimality conditions \cref{optcond} follows from \cref{optcondproof2} with $\bar y=\Lambda\bar x$. The equality conditions follow similarly by $\bar y=\Lambda\bar x$ and the first equation in \cref{optcondproof2}.
\end{proof}
%
\begin{remark}
We remark that  \cref{thmnecopt}  still hold  when considering  \cref{foptprob} in the  infinite dimensional sequence spaces $\ell^p$ in the case $\Lambda=I$.
\end{remark}
Moreover, we have the following corollary, which can be proved as in \cite{KI}, Corollary $2.1$.
\begin{corollary}\label{lowerbound}
	If $(\Lambda\bar x)_i \neq 0$, then $|(\Lambda\bar x)_i| \geq \left( \frac{2 \beta (1-p)}{|(A\Lambda^{-1})_i|_2^2}\right)^{\frac{1}{2-p}}.$
\end{corollary}
\subsection{The augmented Lagrangian formulation and the primal-dual active set strategy}
The active set strategy can be motivated by the following augmented Lagrangian formulation for problem \cref{foptprob}.  Let $P$ be a nonnegative self-adjoint matrix $P$, satisfying 
\begin{equation}\label{P}
((A^TA+\eta P)x,x)\geq \xi |x|_{2}^2
\end{equation}
for some $\eta,\xi>0$, independent of $x \in \R^n$. We set
\begin{equation}\label{bonchoix}
B_i=|(\bar A \Lambda^{-1})_i|_2^2, \,\, \mbox{where }\,\, \bar{A}=\begin{pmatrix}A&\\(\eta P)^{\frac{1}{2}}\end{pmatrix},
\end{equation}
and let $B$ denote the diagonal invertible operator with entries $B_i$. 
Thus, if $A$ is nearly singular, we use $\eta>0$ and the  functional $\frac{\eta}{2}(x,Px)$ to regularize \cref{foptprob}. Consider the associated augmented Lagrangian functional
$$
L(x,y,\lambda)=\frac{1}{2}|Ax-b|_2^2+\frac{\eta}{2}(Px,x)+\beta\sum_{i=1}^n|y_i|^p+\sum_{i=1}^n\frac{B_i}{2}|y_i-(\Lambda x)_i|^2+(\lambda_i,(\Lambda x)_i-y_i).
$$
Given $x,\lambda$, we first minimize the Lagrangian $L$ coordinate-wise with respect  to $y$. For this purpose we consider
\begin{eqnarray}
& &\beta |y_i|^p+\frac{B_i}{2}|y_i-(\Lambda x)_i|^2-(\lambda_i,(\Lambda x)_i-y_i) \nonumber \\ &=& \beta |y_i|^p+\frac{B_i}{2}\left(y_i^2-2y_i\left((\Lambda x)_i+\frac{\lambda_i}{B_i}\right)\right)+\frac{B_i(\Lambda x)_i^2}{2}+\lambda_i(\Lambda x)_i \nonumber\\ &=& \beta |y_i|^p+\frac{B_i}{2}\left[y_i-\left((\Lambda x)_i+\frac{\lambda_i}{B_i}\right)\right]^2-\frac{B_i}{2}\left[(\Lambda x)_i+\frac{\lambda_i}{B_i}\right]^2+\frac{B_i(\Lambda x)_i^2}{2}+\lambda_i(\Lambda x)_i \nonumber \\ &=& \beta |y_i|^p+\frac{1}{2}\left[B_i^{\frac{1}{2}}y_i-\left(B_i^{\frac{1}{2}}(\Lambda x)_i+\frac{\lambda_i}{B_i^{\frac{1}{2}}}\right)\right]^2-\frac{\lambda_i^2}{2B_i}.
\end{eqnarray}
Then, by  \cref{thmnecopt}, the Lagrangian $L$ can be minimized coordinate-wise with respect to $y$ by considering the expressions $\beta |y_i|^p+\frac{1}{2}\left[B_i^{\frac{1}{2}}y_i-\left(B_i^{\frac{1}{2}}(\Lambda x)_i+\frac{\lambda_i}{B_i^{\frac{1}{2}}}\right)\right]^2$ to obtain
\begin{equation}\label{mapphi}
y_i=\Phi(x,\lambda)_i=\left\{
\begin{array}{ll}
|y_i|>0,\,\, B_iy_i+\frac{\beta p y_i}{|y_i|^{2-p}}=B_i(\Lambda x)_i+\lambda_i  \quad &\mbox{ if }\,\, \left|B_i(\Lambda x)_i +\lambda_i\right|>\mu_i\\
0 \quad & \mbox{ otherwise},
\end{array}
\right.\,
\end{equation}
where $\mu_i=\beta^{\frac{1}{2-p}}(2-p)(2(1-p))^{-\frac{1-p}{2-p}}B_i^{\frac{1-p}{2-p}}$. \\ 
Given $y,\lambda$, we minimize $L$  at $x$ to obtain
$$
A^*(Ax-b)+\eta P x+\Lambda^*B(\Lambda x-y)+\Lambda^*\lambda=0,
$$
where $B$ is the diagonal operator with entries $B_i$. Thus, the augmented Lagrangian method \cite{KIb} uses the updates:
\begin{equation}
\begin{cases}
A^*(Ax^{n+1}-b)+\eta P x^{n+1}+\Lambda^*B(\Lambda x^{n+1}-y^n)+\Lambda^*\lambda^n=0,\\
y^{n+1}=\Phi(x^{n+1},\lambda^n),\\
\lambda^{n+1}=\lambda^n+B(\Lambda x^{n+1}-y^{n+1}).
\end{cases}
\end{equation}
If it converges, i.e. $x^n\rightarrow x, y^n \rightarrow \Lambda x^n$ and $\lambda^n\rightarrow \lambda$, then
\begin{equation}\label{mapphi}
\left\{
\begin{array}{lll}
A^*(Ax-b)+\eta Px+\Lambda^*\lambda=0,\\
(\Lambda x)_i=0 \quad &\mbox{ if }\,\, \left|B_iy_i +\lambda_i\right|\leq\mu_i,\\
|(\Lambda x)_i|>0 \mbox{ and } \lambda_i=\frac{\beta p (\Lambda x)_i}{|(\Lambda x)_i|^{2-p}}, \quad &\mbox{ if }\,\, \left|B_iy_i +\lambda_i\right|>\mu_i,
\end{array}
\right.\,
\end{equation}
which is the optimality condition for  $J_P(x)=\min_{x \in \R^n} \frac{1}{2}|Ax-b|_2^2+\beta|\Lambda x|_p^p+\frac{\eta}{2}(x,Px)$. \\

Motivated by  the form of the optimality conditions \cref{mapphi} obtained by the augmented Lagrangian formulation, we formulate a  primal-dual active set strategy for the following system 
\begin{equation}\label{mapphieps}
\left\{
\begin{array}{lll}
A^*(Ax-b)+\eta Px+\Lambda^*\lambda=0,\\
(\Lambda x)_i=0 \quad &\mbox{ if }\,\, \left|B_iy_i +\lambda_i\right|\leq\mu_i,\\
\lambda_i=\frac{\beta p (\Lambda x)_i}{\max(\eps^{2-p},|(\Lambda x)_i|^{2-p}|)}, \quad &\mbox{ if }\,\, \left|B_iy_i +\lambda_i\right|>\mu_i,
\end{array}
\right.\,
\end{equation}
where $\eps>0$ in the third equation is a  fixed parameter enough small. Note that \cref{mapphieps} coincides with \cref{mapphi} for $\eps=0$. The scope of the parameter $\eps$  is to avoid the computation of $\frac{\beta p (\Lambda x^{n+1})_i}{|(\Lambda x^{n+1})_i|^{2-p})|}$  when $(\Lambda x^{n+1})_i=0$, which could happen if $x^{n+1}$ is   far enough from a solution of the optimality conditions. 

\begin{algorithm}[h!]
	\caption{Primal-dual active set strategy}
	\begin{algorithmic}[1]
		\STATE Initialize $\lambda^0, x^0$.  Set $y^0=\Lambda x^0$. Set $n=0$.
		\REPEAT 
\STATE Solve for $(x^{n+1}, \lambda^{n+1})$
\begin{equation}\label{eqnumstr1}
A^*(Ax^{n+1}-b)+\eta Px^{n+1}+ \Lambda^*\lambda^{n+1}=0,
\end{equation}
where
\begin{align}
(\Lambda x^{n+1})_i=0 \quad &\mbox{ if }\,\, i \in \{i \, : \, |B_iy^n_i+\lambda^n_i|\leq \mu_i\}\label{active1}\\
\lambda^{n+1}_i=\frac{\beta p (\Lambda x^{n+1})_i}{\max(\eps^{2-p},|(\Lambda x^{n+1})_i|^{2-p})} \quad  &\mbox{ if }\,\, i\in \{i \, : \, |B_iy^n_i+\lambda^n_i|> \mu_i\}\label{inactive1}.
\end{align}
\STATE Set $y^{n+1}=\Lambda x^{n+1}$, $n=n+1$.
\UNTIL{the stopping criterion is fulfilled.}
	\end{algorithmic}
\end{algorithm}

\begin{remark}\label{remc}
Note that $B_i$ has to be chosen exactly as in \eqref{bonchoix} in order to have the convergence of the method  to the optimality condition \cref{optcond}. In  \cite{LI} an alternate direction method of multipliers is proposed for problems as in \cref{foptprob} and the augmented Lagrangian formulation is considered with a penalization term  chosen "large enough". The convergence of the proposed alternate direction method of multiplier is proved to a stationary point as defined in \cite{LI}, equation $4$. We deduce that, due to the different choice in the penalization term, the stationary points considered in \cite{LI} (to which the ADMM proposed in \cite{LI} is proved to converge) do not coincide with the stationary points identified by our optimality condition \cref{optcond}. To make it evident, we propose to look at the following $1$-dimensional example.
Suppose we want to minimize
\begin{equation}\label{problemca}
\frac{1}{2}|x-b|_2^2+\beta|x|_p^p
\end{equation}
for  $p \in (0,1],$ $ \beta >0$. 
 By  \cref{thmnecopt}, the optimality condition is 
\begin{equation}\label{optcondcx}
\left\{
\begin{array}{lll}
x=0\quad &\mbox{ if } b<\mu\\
|x|>0 \mbox{ and } x-b+\beta p\frac{x}{|x|^{2-p}}=0 \quad &\mbox{ if } b >\mu,
\end{array}
\right.\,
\end{equation}
where we denote $\mu:=d_{\beta,p}$ and $d_{\beta,p}=\beta^{\frac{1}{2-p}}(2-p)(2(1-p))^{-\frac{1-p}{2-p}}$ is given in \cref{optcond}. 
Consider for $c>0$ the augmented Lagrangian 
\begin{equation}\label{auglagc}
L(x,y)=\frac{1}{2}|x-b|_2^2+\beta|y|^p+\frac{c}{2}|x-y|^2_2.
\end{equation}
Given $y$ fixed, we minimize with respect to $x$ to obtain
$$
x-b+c(x-y)=0.
$$
Then, given $x$ fixed, we minimize with respect to $y$  the expression $\beta|y|^p+\frac{c}{2}|x-y|^2$. 
By \cref{thmnecopt}, we obtain
\begin{equation}\label{optcondcy}
\left\{
\begin{array}{lll}
y=0\quad &\mbox{ if } cx<\mu_c\\
c(y-x)+ \beta p\frac{y}{|y|^{2-p}}=0 \quad &\mbox{ if } cx >\mu_c,
\end{array}
\right.\,
\end{equation}
where $\mu_c=d_{\beta,p}\sqrt{c}^{\frac{(2-2p)}{2-p}}$. Then we obtain the following optimality conditions 
\begin{equation}\label{optcondc}
\left\{
\begin{array}{lll}
x-b+c(x-y)=0\\
y=0\quad &\mbox{ if } cx<\mu_c\\
c(y-x)+\beta p\frac{y}{|y|^{2-p}} =0\quad &\mbox{ if }cx >\mu_c.\\
\end{array}
\right.\,
\end{equation}
Note that if $c>1$, then $\mu<\mu_c$. Then we consider $\mu<b<\mu_c$ in the augmented Lagrangian formulation \cref{auglagc} and  we get that 
$
y=0, x=\frac{b}{1+c} 
$
is a solution to \cref{optcondc}
and we note that
$$
(x,y)\to (0,0)\, \, \mbox{ as } c \to + \infty.
$$
On the contrary, since $b>\mu$, we have that $x=0$  is not a solution of \cref{optcondcx}.  We remark that considering a Lagrange multiplier in \cref{auglagc} leads to the same conclusion.
\end{remark}

\subsection{Convergence of the primal-dual active set strategy: case $\Lambda=I$}
While the numerical performance of the primal-dual active set strategy proved to be very successful, its convergence analysis is still a substantial challenge. Then we focus on the case $\Lambda=I$ for which we can give a sufficient condition for uniqueness of the solution to \cref{mapphieps} and for convergence. Moreover, the case $\Lambda=I$ will be successfully tested in an image recontruction problem in  \cref{mimrec}.\\
%
\begin{remark}
We remark that the uniqueness and the convergence results, namely  \cref{uniqueness} and  \cref{convddc},  still hold when considering optimization of problems as \cref{foptprob} in the  infinite dimensional sequence spaces $\ell^p$.
\end{remark}
\begin{remark}\label{boundsol}
Notice that the sequence $\{x^n\}_{n\in \mathbb{N}}$ is bounded uniformly in $n$. Indeed since $(x^{n+1}, \lambda^{n+1})\geq 0$ for all $n$,  we have from the first equation in \cref{mapphieps}
$$
((A^*A+\eta P)x^{n+1}, x^{n+1}) \leq (Ax^{n+1},b),
$$
which coupled with  \eqref{P} gives
$
\xi |x^{n+1}|_2\leq \|A\|_2|b|_2.
$
We denote $M:=\|A\|_2|b|_2\xi^{-1}$, where $\xi$ is defined in \cref{P}. Then 
$
|x^{n+1}|_2\leq M.
$
\end{remark}
\subsubsection{Uniqueness}\label{subuniq}
For any pair $x,\lambda$ we define
$$
\mathcal{I}(x,\lambda)=\{i \, : \, |B_ix_i+\lambda_i|> \mu_i\} \mbox{ \,\,and\,\, } \mathcal{A}(x,\lambda)=\{i \, : \, |B_ix_i+\lambda_i|\leq \mu_i\} 
$$
and we set
$$
Q=A^*A+\eta P.
$$
We denote for $p \in (0,1]$
\begin{equation}\label{ag}
\alpha:=\frac{1-p}{p-2}\leq 0, \quad \gamma=\frac{1}{p-2}<0,
\end{equation}
and we note that
\begin{equation}\label{apiug}
\alpha+1=-\gamma.
\end{equation}
We will use the following diagonal dominance condition:
\begin{equation}\label{ddc}
\|B^{\alpha}(Q-B)B^{\gamma}|ì\|_\infty\leq \rho \mbox{ for some } \rho \in (0,1).
\end{equation}
\begin{remark}
In the case that $Q$ is a diagonal matrix $Q-B=0$ and \cref{ddc} is trivially satisfied.
\end{remark}
\begin{remark}
We observe that for $p\to0$, we have $\alpha=\gamma=-\frac{1}{2}$. In particular \cref{ddc} coincides with the diagonal dominance condition considered in \cite{KI} to prove the convergence of the primal dual active set strategy in the case $p=0$.
 \end{remark}
We set the following notation which will be used for the rest of this section:
 \begin{equation}\label{constants}
 C=(2-p)(2(1-p))^{-\frac{1-p}{2-p}},\quad E=p \|B^{\alpha}\|_\infty  |x|_\infty, \quad F=|x|_\infty \|B^{-\gamma}\|_\infty.
 \end{equation}
 Under the diagonal dominance condition \cref{ddc}, 
we prove that, if $x, \lambda$ is a solution to \cref{mapphieps} satisfying one of the following conditions
 \begin{equation}\label{cu}
\min_{\mathcal{I}(x,\lambda)}|B^{\alpha}(\lambda+Bx)|\geq (1+\delta)\beta^{-\gamma}C,
\end{equation}
\begin{equation}\label{cu2}
\max_{\mathcal{A}(x,\lambda)}|B^{\alpha}(\lambda+Bx)|\leq (1-\delta)\beta^{-\gamma}C,
\end{equation}
for some $\delta >0$ large enough, then it is necessarely unique.\\ Above $\min_{\mathcal{I}(x,\lambda)}|B^{\alpha}(\lambda+Bx)|$ stands for $\min_{i \in \mathcal{I}(x,\lambda)}|B_i^{\alpha}(\lambda_i+B_i x_i)|$.
Henceforth we refer to  \cref{cu}-\cref{cu2} as strictly complementary condition.   Note that $\mu_i=\beta^{-\gamma}CB_i^{-\alpha}$.\\
The precise statement of the uniqueness result  is given in the following theorem. The proof is inspired by \cite{KI}, Theorem $5.1$.

\begin{theorem}\label{uniqueness}(Uniqueness)
Assume that \cref{ddc} holds. Let $C, E, F$ be defined as in \cref{constants} and $\alpha, \gamma$ in \cref{ag}.
 Then there exists at most one solution to \cref{mapphieps} satisfying \cref{cu} 
for some $\delta>0$ large enough and depedending on $\eps, \rho, \beta,\alpha,\gamma, C, E, F$ (see \cref{delta}).
An analogous statement holds with \cref{cu} replaced by \cref{cu2}. 
\end{theorem}

	\begin{remark}\label{delta}
	More precisely, it will be seen from the proof that  $\delta$ in \cref{cu} has to  satisfy
	$
\delta > \frac{2\rho}{1-\rho} (1+ \beta^{-\alpha}\eps^{\frac{1}{\gamma}} E C^{-1})+\beta^{-\alpha} \eps^{\frac{1}{\gamma}} E C^{-1} +\beta^{\gamma}FC^{-1}:=\bar{\delta},
$
where we recall that $-\alpha \geq 0$ and $\gamma<0$.
	\end{remark}
\begin{proof}
Assume that there exist two pairs $x,\lambda$ and $\hat x, \hat \lambda$ satisfying \cref{mapphieps} and \cref{cu}.  Then we have
\begin{equation}\label{11}
Q(x-\hat x)+\lambda -\hat \lambda=0.
\end{equation}
 Multiplying \cref{11} by $B^\alpha$ and using \cref{apiug}, we have
\begin{equation}\label{u1}
B^{-\gamma}x+B^{\alpha}\lambda-(B^{-\gamma}\hat x+B^{\alpha}\hat \lambda)=B^{\alpha}(B-Q)B^{\gamma}B^{-\gamma}(x-\hat x).
\end{equation}
\textbf{Case 1:} First consider the case $x_i \neq 0$ if and only if $\hat x_i \neq 0$. Then we find
\begin{equation}\label{lambdak}
\lambda_i=\frac{\beta p x_i}{\max(\eps^{2-p}, |x_i|^{2-p})}, \quad \hat{\lambda}_i=\frac{\beta p \hat{x}_i}{\max(\eps^{2-p}, |\hat{x}_i|^{2-p})}.
\end{equation}
Equations \cref{u1} and \cref{lambdak} and the diagonal dominance condition \cref{ddc} imply that 
$$
B_i^{-\gamma}(x_i-\hat{x}_i)+B_i^{\alpha}\left(\frac{\beta p x_i}{\max(\eps^{2-p}, |x_i|^{2-p})}-\frac{\beta p \hat{x}_i}{\max(\eps^{2-p}, |\hat{x}_i|^{2-p})}\right)\leq \rho |B^{-\gamma}(x-\hat x)|_\infty
$$
and hence  we have
\begin{equation}\label{u2}
|B^{-\gamma}(x-\hat x)|_\infty \leq \frac{2\beta \eps^{\frac{1}{\gamma}} E}{1-\rho},
\end{equation}
where $E$ is defined in \cref{constants}.
Then by \cref{apiug}, \cref{u1}, \cref{ddc} and \cref{u2}, we have for each $i$:
\begin{eqnarray*}
|B_i^{\alpha}(\lambda_i+B_i x_i)|-|B_i^{\alpha}(\hat \lambda_i+B_i \hat x_i)|&\leq &|B_i^{\alpha}(\lambda_i-\hat \lambda_i)+B_i^{-\gamma}(x_i-\hat x_i)|\nonumber \\ &\leq & |B^{\alpha}(\lambda-\hat \lambda)+B^{-\gamma}(x -\hat x)|_\infty\nonumber =  \|B^{\alpha}(B-Q)B^{\gamma}B^{-\gamma}(x-\hat x)\|_\infty \\&\leq&\rho |B^{-\gamma}(x-\hat x)|_\infty \leq   \frac{ 2\rho \beta \eps^{\frac{1}{\gamma}} E }{1-\rho}
\end{eqnarray*}
and thus
\begin{equation}\label{refstrana}
|B_i^{\alpha}(\lambda_i+B_i x_i)|-|B_i^{\alpha}(\hat \lambda_i+B_i \hat x_i)|\leq   \frac{ 2\rho \beta \eps^{\frac{1}{\gamma}} E }{1-\rho}.
\end{equation}
By \cref{apiug} and the second equation in \cref{lambdak} we get
\begin{equation}\label{ub1}
|B_i^{\alpha}(\hat \lambda_i+B_i\hat x_i)|\leq |B_i^{\alpha}\hat \lambda_i|+|B_i^{-\gamma}\hat x_i|\leq \beta \eps^{\frac{1}{\gamma}} E+F,
\end{equation}
where $E,F$ are defined in \cref{constants}.
From \cref{refstrana}, \cref{ub1} and \cref{cu} we deduce that
$$
(1+\delta)\beta^{-\gamma}C-\beta \eps^{\frac{1}{\gamma}} E-F\leq \frac{2\rho \beta \eps^{\frac{1}{\gamma}} E}{1-\rho},
$$
hence
$
\delta \leq\frac{2\rho \beta^{-\alpha} \eps^{\frac{1}{\gamma}} E C^{-1}}{1-\rho}+\beta^{-\alpha} \eps^{\frac{1}{\gamma}} E C^{-1} +\beta^{\gamma}FC^{-1},
$
 and for $\delta >\frac{2\rho \beta^{-\alpha} \eps^{\frac{1}{\gamma}} E C^{-1}}{1-\rho}+\beta^{-\alpha} \eps^{\frac{1}{\gamma}} E C^{-1} +\beta^{\gamma}FC^{-1}$, we get a contradiction.\\
\textbf{Case 2:} Suppose there exists $j$ such that $\mbox{sign}|x_j|\neq \mbox{sign}|\hat x_j|$. Without loss of generality we can assume that $x_j\neq 0$ and $\hat x_j=0$. Note that from the definition of the active set $\mathcal{A}$ and the last equation in \cref{mapphieps} we have
\begin{equation}\label{case2}
|B^{\alpha}_i\hat \lambda_i|\leq \beta^{-\gamma}C \mbox{ if } \hat x_i=0,\quad |B^{\alpha}_i\hat \lambda_i|\leq \beta  \eps^{\frac{1}{\gamma}} E \mbox{ if } \hat x_i\neq 0
\end{equation}
and similarly
\begin{equation}\label{case22}
|B^{\alpha}_i\lambda_i|\leq \beta^{-\gamma}C \mbox{ if }  x_i=0, \quad |B^{\alpha}_i\lambda_i|\leq \beta \eps^{\frac{1}{\gamma}} E \mbox{ if } x_i\neq 0.
\end{equation}
Then by \cref{u1} and \cref{ddc}  we have 
$
B_i^{-\gamma}(x_i-\hat{x}_i)\leq B_i^{\alpha}(\hat\lambda_i- \lambda_i)+\rho |B^{-\gamma}(x-\hat x)|_\infty
$
and by \cref{case2} and \cref{case22}, we get
$$
|B^{-\gamma}(x-\hat x)|_\infty \leq \frac{2(\beta^{-\gamma}C+\beta \eps^{\frac{1}{\gamma}} E)}{1-\rho}.
$$
Then using again \cref{u1} for $j$ chosen as above and proceeding  as in \textbf{Case 1}, we get
\begin{equation}\label{3}
|B_j^{\alpha}(\lambda_j-B_jx_j)|-|B_j^{\alpha}(\hat{\lambda}_j-B_j\hat{x}_j)|\leq \frac{2\rho (\beta^{-\gamma}C+\beta \eps^{\frac{1}{\gamma}} E)}{1-\rho}.
\end{equation}
By the first equation in \cref{case2}, \cref{3}, \cref{cu}, we get
$$
(1+\delta) \beta^{-\gamma}C-\beta^{-\gamma}C\leq \frac{2\rho (\beta^{-\gamma}C+ \beta \eps^{\frac{1}{\gamma}} E)}{1-\rho}
$$
and hence we have a contradiction by taking $\delta > \frac{2\rho}{1-\rho} (1+ \beta^{-\alpha} \eps^{\frac{1}{\gamma}}EC^{-1})$.
The case $\max_{\mathcal{A}(x,\lambda)}|B^{\alpha}(\lambda+Bx)|\leq (1-\delta) \beta^{-\gamma}C$ can be treated analogously.
\end{proof}

\subsubsection{Convergence: Diagonal dominant case}\label{subconv}
Here we give a sufficient condition for the convergence of the primal-dual active set method. Following the ideas of \cite{KI} (in particular Proposition $5.1$), we utilize the diagonal dominance condition \cref{ddc} and consider a solution $x, \lambda$ to \cref{mapphieps} which satisfies the strict complementary condition. As such it is unique according to \cref{uniqueness}. We use the same notation as in  \cref{subuniq}.
\begin{proposition}\label{convddc}
Let $C, E, F$ be defined as in \cref{constants} and $\alpha, \gamma$ as in \cref{ag}. Suppose that \cref{ddc} holds. Let $\bar x, \bar \lambda$ be a solution to \cref{mapphieps} satisfying the strict complementary condition \cref{cu}-\cref{cu2},
for some $\delta$ large enough depending on $\eps,\rho,\beta, \alpha,\gamma, C,E,F$ (see  \cref{delta2}). Then the sets
$$
\mathcal{S}^n=\left\{i \in \mathcal{I}(\bar x, \bar \lambda)\, : \, \lambda^n_i=\frac{\beta p x_i^n}{\max(\eps^{2-p}, |x_i^n|^{2-p})}\right\}, \quad \mathcal{T}^n=\{i \in \mathcal{A}(\bar x, \bar \lambda) \, : \, x^n_i=0\}
$$
are monotonically nondecreasing. As soon as $\mathcal{S}^n=\mathcal{S}^{n+1}$ and $\mathcal{T}^n=\mathcal{T}^{n+1},$ then for some n, we have $(x^n, \lambda^n)=(\bar x, \bar \lambda)$.
\end{proposition}
\begin{remark}\label{delta2}
More specifically, it will be seen from the proof that   $\delta$ has to satisfy
 $
 \delta >\frac{\rho(2 \rho\beta^\gamma FC^{-1}+2\beta^{-\alpha} E\eps^{\frac{1}{\gamma}}C^{-1}+1)}{1-\rho}+3\beta^{-\alpha}E\eps^{\frac{1}{\gamma}}C^{-1}+\bar{\delta}:=\tilde \delta,
 $
 where $\bar{\delta}$ is defined in \cref{delta}.
\end{remark}
\begin{proof}
We divide the proof into three steps. In \textbf{Step} (i) we prove an estimate  which will be used throughout the rest of the proof, in \textbf{Step} (ii) we prove the monotonicity of $\mathcal{S}^n$ and $\mathcal{T}^n$ and finally in \textbf{Step} (iii) we conclude the proof of  convergence.
\begin{step1}
\upshape
We  have
\begin{equation}\label{eqnhat0}
Q(x^n-\bar x)+\lambda^n-\bar \lambda=0.
\end{equation}
Multiplying \cref{eqnhat0} by $B^\alpha$ and using \cref{apiug} we get
\begin{equation}\label{eqnhat}
B^{-\gamma}(x^n-\bar x)+B^{\alpha}(\lambda^n-\bar \lambda)=B^{\alpha}(B-Q)B^{\gamma}B^{-\gamma}(x^n-\bar x).
\end{equation}
We consider separately the cases $x^n_i=0, \bar x_i\neq 0$, and $x^n_i\neq 0, \bar x_i\neq 0$, and $x^n_i\neq 0, \bar x_i=0$. First consider $x^n_i=0, \bar x_i \neq 0$.
For two consecutives iterated we have
$$
Q(x^n-x^{n-1})+\lambda^n-\lambda^{n-1}=0
$$
and thus, multiplying the equation by $B^{\alpha}$ and using \cref{apiug}, we get
\begin{equation}\label{c1}
B^{\alpha}(\lambda^n+Bx^n)-B^{\alpha}(\lambda^{n-1}+Bx^{n-1})=B^{\alpha}(B-Q)B^{\gamma}B^{-\gamma}(x^n-x^{n-1}).
\end{equation}
Since $x^n_i=0$, then $|B_i x_i^{n-1}+\lambda_i^{n-1}|\leq \mu_i=\beta^{-\gamma}CB_i^{-\alpha}$, and  by \cref{c1} and \cref{ddc} we get
\begin{eqnarray*}
|B^{\alpha}_i\lambda^n_i|\leq |[B^{\alpha}(B-Q)B^{\gamma}B^{-\gamma}(x^n-x^{n-1})]_i| \nonumber &+&|B_i^{\alpha}(\lambda_i^{n-1}+B_ix_i^{n-1})| \\ \nonumber &\leq &\rho |B^{-\gamma}(x^n-x^{n-1})|_\infty +\beta^{-\gamma}C\\  &\leq & 2 \rho F +\beta^{-\gamma}C.
\end{eqnarray*}
Since $\bar x_i \neq 0$, by \cref{mapphieps}  we have
$
 |B^{\alpha}_i\bar \lambda_i|\leq \beta \eps^{\frac{1}{\gamma}} E.
$
By the previous estimates, \cref{eqnhat} and  \cref{ddc}, we get
\begin{equation}\label{c2}
|B_i^{-\gamma}(x_i^n-\bar x_i)|\leq 2\rho F + \beta^{-\gamma}C+\beta \eps^{\frac{1}{\gamma}} E+\rho |B^{-\gamma}(x^n-\bar x)|_\infty.
\end{equation}
If $x^n_i\neq 0$ and $\bar x_i\neq 0$, the update rule of the algorithm and \cref{mapphieps} imply
\begin{equation}\label{c3}
|B_i^{-\gamma}(x_i^n-\bar x_i)|\leq 2\beta \eps^{\frac{1}{\gamma}} E+\rho|B^{-\gamma}(x^n-\bar x)|_\infty.
\end{equation}
Similarly if $x^n_i\neq 0$ and $\bar x_i=0$, we get
\begin{equation}\label{c4}
|B_i^{-\gamma}(x_i^n-\bar x_i)|\leq \beta^{-\gamma}C+\beta \eps^{\frac{1}{\gamma}} E +\rho |B^{-\gamma}(x^n-\bar x)|_\infty.
\end{equation}
Then by \cref{c2}, \cref{c3} and \cref{c4}, we have
$$
|B_i^{-\gamma}(x_i^n-\bar x_i)|\leq 2\rho F+\beta^{-\gamma}C +2\beta \eps^{\frac{1}{\gamma}} E+\rho |B^{-\gamma}(x^n-\bar x)|_\infty
$$
and then 
\begin{equation}\label{c5}
|B^{-\gamma}(x^n-\bar x)|_\infty\leq \frac{2 \rho F+\beta^{-\gamma}C+2\beta E\eps^{\frac{1}{\gamma}}}{1-\rho}=\frac{\tilde{A}}{1-\rho},
\end{equation}
where we denote $\tilde{A}=2 \rho F+\beta^{-\gamma}C+2\beta \eps^{\frac{1}{\gamma}}E$.\\
\end{step1}
\begin{step2}
\upshape
We consider \cref{eqnhat} on $\mathcal{S}^n$. By \cref{cu}, \cref{cu2}, \cref{c5} and the definition of $\mathcal{S}^n=\left\{\lambda^n_i=\frac{\beta p x_i^n}{\max(\eps^{2-p}, |x_i^n|^{2-p})}, \bar \lambda_i=\frac{\beta p \bar x_i}{\max(\eps^{2-p}, |\bar x_i|^{2-p})}\right\}$,  we get
$$
\max_{\mathcal{S}^n}|B^{-\gamma}_i(x_i^n-\bar x_i)|\leq 2\beta \eps^{\frac{1}{\gamma}} E+\rho |B^{-\gamma}(x^n-\bar x)|_\infty \leq 2\beta E\eps^{\frac{1}{\gamma}}+\frac{\rho \tilde{A}}{1-\rho}.
$$
For $i \in \mathcal{S}^n$, by the complementary condition  $|B^{\alpha}_i(\bar \lambda_i +B_i\bar x_i)|>(1+\delta)\beta^{-\gamma}C$, \cref{apiug} and the definition of $\mathcal{S}^n$, we have
\begin{equation}\label{h1}
|B^{-\gamma}_i\bar x_i|>(1+\delta)\beta^{-\gamma}C-\beta \eps^{\frac{1}{\gamma}}E.
\end{equation}
Then by \cref{apiug}, \cref{h1} and \cref{c5} we get
\begin{eqnarray*}
|B_i^\alpha(\lambda^n_i+B_ix^n_i)|=|B_i^{\alpha}\lambda_i^n+B_i^{-\gamma}x_i^n|\geq |B_i^{-\gamma}\bar x_i|&-&|B_i^{\alpha}\frac{\beta p x_i^n}{\max(\eps^{2-p}, |x_i^n|^{2-p})}+B_i^{-\gamma}(x_i^n-\bar x_i)|\\&>& (1+\delta)\beta^{-\gamma}C-3\beta \eps^{\frac{1}{\gamma}}E-\frac{\rho \tilde{A}}{1-\rho}.
\end{eqnarray*}
Notice that  by taking $\delta>3\beta^{-\alpha}  \eps^{\frac{1}{\gamma}}EC^{-1}+\frac{\rho \tilde{A}\beta^{\gamma}C^{-1}}{1-\rho}$, we get
$$
|B_i^\alpha(\lambda^n_i+B_ix^n_i)|> \beta^{-\gamma}C.
$$
Then by \cref{inactive1} we have $\lambda_i^{n+1}=\frac{\beta p x_i^{n+1}}{\max(\eps^{2-p}, |x_i^{n+1}|^{2-p})}$, and $i \in \mathcal{S}^{n+1}$ and $\mathcal{S}^n\subseteq \mathcal{S}^{n+1}$ follow.\\
For $i \in \mathcal{T}^n$ by \cref{eqnhat}, \cref{ddc} and \cref{c5} we have
\begin{equation}\label{h2}
|B_i^{\alpha}(\lambda_i^n-\bar \lambda_i)|\leq \rho |B^{-\gamma}(x^n-\bar x)|_\infty\leq \frac{\rho \tilde{A}}{1-\rho},
\end{equation}
and  by the definition of $\mathcal{T}^n$,  the  complementary condition  $|B_i^{\alpha}\bar \lambda_i|\leq (1-\delta)\beta^{-\gamma} C$  and \cref{h2}, we get
$$
|B_i^{\alpha}(\lambda_i^n+B_ix_i^n)|=|B^{\alpha}\lambda_i^n|\leq |B_i^{\alpha}(\lambda_i^n-\bar \lambda_i)|+|B_i^{\alpha}\bar \lambda_i|\leq \frac{\rho \tilde{A}}{1-\rho}+ (1-\delta)\beta^{-\gamma} C< \beta^{-\gamma}C,
$$
where the last inequality holds by taking $\delta> \frac{\rho \tilde{A}\beta^{\gamma}C^{-1}}{1-\rho}$. Then for such $\delta$ we have
$$
|B_i^{\alpha}(\lambda_i^n+B_ix_i^n)|<\beta^{-\gamma}C 
$$
and hence $x_i^{n+1}=0$ and $i \in \mathcal{T}^{n+1}$. Thus $\mathcal{T}^n\subseteq \mathcal{T}^{n+1}$.\\
\end{step2}
\begin{step3}
\upshape
Assume $\mathcal{S}^n=\mathcal{S}^{n+1} \subset \mathcal{I}(\bar x, \bar \lambda)$ and $\mathcal{T}^n=\mathcal{T}^{n+1}\subset \mathcal{A}(\bar x, \bar \lambda)$ and
$$
\mathcal{S}^n\cup \mathcal{T}^n\subsetneq \mathcal{I}(\bar x, \bar \lambda) \cup \mathcal{A}(\bar x, \bar \lambda).
$$
Assume $i \in \mathcal{A}(\bar x, \bar \lambda) \backslash \mathcal{T}^n$. Then
\begin{equation}\label{equationstn1}
x_i^{n+1}\neq 0, \quad x_i^n\neq 0, \quad \bar x_i=0,
\end{equation}
\begin{equation}\label{equationstn2} 
\lambda_i^{n+1}=\frac{\beta p x_i^{n+1}}{\max(\eps^{2-p}, |x_i^{n+1}|^{2-p})}, \quad \lambda_i^{n}=\frac{\beta p x_i^{n}}{\max(\eps^{2-p}, |x_i^{n}|^{2-p})}.
\end{equation}
By \cref{eqnhat}, the last equation in \cref{equationstn2},  the  complementary condition $|B^{\alpha}_i\bar \lambda_i|\leq (1-\delta)\beta^{-\gamma}C$,  \cref{ddc} and \cref{c5}, we get
\begin{equation}\label{c7}
|B_i^{-\gamma}(x_i^n-\bar x_i)|\leq \beta \eps^{\frac{1}{\gamma}} E +(1-\delta)\beta^{-\gamma}C+\rho |B^{-\gamma}(x^n-\bar x)|_\infty  \leq \beta \eps^{\frac{1}{\gamma}}E +(1-\delta)\beta^{-\gamma}C+\frac{\rho \tilde{A}}{1-\rho}.
\end{equation}
The first equation in \cref{equationstn1} and the last in \eqref{equationstn2}, \cref{apiug} and the update rule of the algorithm imply
\begin{equation}\label{c6}
|B^{-\gamma}_ix^n_i|\geq \beta^{-\gamma}C-\beta \eps^{\frac{1}{\gamma}} E, 
\end{equation}
Thus by \cref{c6}, \cref{c7} and the third equation in \cref{equationstn1} we have
$$
\beta^{-\gamma}C-\beta \eps^{\frac{1}{\gamma}}E\leq \beta \eps^{\frac{1}{\gamma}}E+(1-\delta)\beta^{-\gamma}C+\frac{\rho \tilde{A}}{1-\rho}
$$
and  we get a contradiction by taking
$
\delta > \frac{\rho \tilde{A}\beta^{\gamma}C^{-1}}{1-\rho} +2\beta^{-\alpha}\eps^{\frac{1}{\gamma}}EC^{-1}.
$
\\
If $i \in \mathcal{I}(\bar x, \bar \lambda) \backslash \mathcal{S}^n$, we have
\begin{equation}\label{equationssn}
\lambda_i^{n+1}\neq\frac{\beta p x_i^{n+1}}{\max(\eps^{2-p}, |x_i^{n+1}|^{2-p})}, \quad \lambda_i^{n}\neq\frac{\beta p x_i^{n}}{\max(\eps^{2-p}, |x_i^{n}|^{2-p})},
\end{equation}
\begin{equation}\label{equationssn2}
\bar \lambda_i=\frac{\beta p \bar x_i}{\max(\eps^{2-p}, |x_i|^{2-p})},\quad x_i^{n+1}=0, \quad x_i^n=0, \quad \bar x_i\neq 0.
\end{equation}
By the first equation in \cref{equationssn}, the third in \cref{equationssn2} and the update rule of the algorithm, we have
\begin{equation}\label{c9}
|B^{\alpha}_i\lambda^n_i|\leq \beta^{-\gamma}C, 
\end{equation}
and by the strict complementary condition $B_i^\alpha(\bar \lambda_i+B_i\bar x_i)>\beta^{-\gamma}C$, \cref{apiug} and the first equation in \cref{equationssn2}, we get
\begin{equation}\label{c10}
|B^{-\gamma}_i\bar x_i|>(1+\delta)\beta^{-\gamma}C-\beta  \eps^{\frac{1}{\gamma}} E.
\end{equation}
By proceeding as in \cref{c7} and using \cref{c9}, we have
$$
 |B_i^{-\gamma}(x_i^n-\bar x_i)|\leq \beta^{-\gamma}C +\beta \eps^{\frac{1}{\gamma}}E+\frac{\rho \tilde{A}}{1-\rho}
$$
and by  \cref{c10}  we get
$$
(1+\delta)\beta^{-\gamma}C<\beta^{-\gamma}C+2\beta \eps^{\frac{1}{\gamma}}E+\frac{\rho \tilde{A}}{1-\rho},
$$
and we have a contradiction by taking 
$
\delta >\frac{\rho \tilde{A}\beta^{\gamma}C^{-1}}{1-\rho}+2\beta^{-\alpha}\eps^{\frac{1}{\gamma}}EC^{-1}.
$
Then $\mathcal{S}^n=\mathcal{I}(\bar x, \bar \lambda)$. Once the active set structure is determined the unique solution is determined by \cref{eqnumstr1}.
\end{step3}
\end{proof}

 


\section{Active set monotone algorithm: numerical results}\label{numericsD}
Here we describe the active set monotone  scheme (see \textbf{Algorithm $3$}) and discuss the numerical results for two different test cases. The first one is the  time-dependent  control problem from  \cref{controlone}, the second one is an example in microscopy image reconstruction. Typically the active set monotone scheme requires fewer iterations and achieves a lower residue than the monotone scheme of  \cref{sec:optcond}.

\subsection{The numerical scheme}
The proposed active set monotone algorithm consists of an \textbf{outer loop} based on the primal-dual active set strategy and an \textbf{inner loop} which uses the monotone algorithm  to solve the nonlinear part of the optimality condition.\\
In order to achieve a better numerical performance, we write the optimality condition as explained in the following. At each iteration of the active-set strategy (\textbf{Algorithm 2}) we solve  the following system in $x^{n+1}, \lambda^{n+1}$
\begin{equation}\label{optcondd}
\left\{
\begin{array}{lll}
A^*(Ax^{n+1}-b)+\eta Px^{n+1}+\Lambda^*\lambda^{n+1}=0\\
(\Lambda  x^{n+1})_i=0 \quad &\mbox{ if }\,\, i \in \mathcal{A}_n\\
\lambda^{n+1}_i=\frac{\beta p (\Lambda x^{n+1})_i}{\max(\eps^{2-p},|(\Lambda x^{n+1})_i|^{2-p})} \quad & \mbox{ if }\,\, i \in \mathcal{I}_n,
\end{array}
\right.\,
\end{equation}
where $\mathcal{A}_n=\{i \, : \, |B_i y^n_i+\lambda^n_i|\leq \mu_i\}$ are the active indexes and $\mathcal{I}_n=\mathcal{A}_n^c$ are the inactive ones. We write \cref{optcondd} in the following form
\begin{equation}\label{sistemalg0}
\left\{
\begin{array}{lll}
(A^*A+\Lambda_{\mathcal{I}_n}^*N_{\mathcal{I}_n}^{n+1} \Lambda_{\mathcal{I}_n}+\eta P) x^{n+1}+ \Lambda_{\mathcal{A}_n}^*\lambda^{n+1}_{\mathcal{A}_n}=A^*b\\
\Lambda_{\mathcal{A}_n} x^{n+1}=0 
\end{array}
\right.\,
\end{equation}
where  $\Lambda_{\mathcal{A}_n}, \Lambda_{\mathcal{I}_n}$ are the rows of $\Lambda$ corresponding to the active and inactive indexes  and $N^{n+1}_{\mathcal{I}_n}$ is the diagonal operator such that $(N^{n+1}_{\mathcal{I}_n})_{ii, i \in \mathcal{I}_n}=\frac{\beta p}{\max(\eps^{2-p},|(\Lambda x^{n+1})_{i \in \mathcal{I}_n}|^{2-p})}$. \\In order to solve  \cref{sistemalg0} we apply the following iterative procedure which is solved for $x^{k+1,n+1}, \lambda^{k+1,n+1}$
\begin{equation}\label{sistemalg1}
\left\{
\begin{array}{lll}
(A^*A+\Lambda_{\mathcal{I}_n}^*N_{\mathcal{I}_n}^{k,n+1} \Lambda_{\mathcal{I}_n}+\eta P) x^{k+1,n+1}+ \Lambda_{\mathcal{A}_n}^*\lambda^{k+1,n+1}_{\mathcal{A}_n}=A^*b\\
\Lambda_{\mathcal{A}_n} x^{k+1,n+1}=0 
\end{array}
\right.\,
\end{equation}
where $N^{k,n+1}_{\mathcal{I}_n}$ is diagonal with $i$-entries $\frac{\beta p}{\max(\eps^{2-p},|(\Lambda x^{k,n+1})_{i \in \mathcal{I}_n}|^{2-p})}$.

\begin{remark}
Note that the system matrix associated to \cref{sistemalg1}
is symmetric.  
\end{remark}
The algorithm stops when the residue of \cref{sistemalg0} and \cref{optcond} (for the inner and the outer cycle respectively) is $O(10^{-12})$ in the control problem and $O(10^{-8})$ in the microscopy image example. \\
We remark that in our numerical tests we  always took $\eta=0$. The  initialization $x^0,\lambda^0$   in the outer cycle is chosen  in the following way
\begin{equation}\label{initMA}
 x^0=(A^*A+2\beta \Lambda^*\Lambda)^{-1}A^*b, \, \, \lambda^0=\Lambda^{-1} A^*(b-A x_0).
\end{equation}
In particular $\lambda^0$ is the solution of the first equation in \cref{optcond} for $x=x^0$. As in \cref{numericsDeps}, for some values of $\beta$ the previous initialization is not suitable. Following the idea already used for the monotone scheme, we successfully tested an analogous  continuation strategy with respect to increasing $\beta$-values. \\
In \textbf{Algorithm $3$} we jump out of at the inner loop in case of presence of singular components. We recall that the singular components are those $i$ such that $|(\Lambda x)_i|<\eps\}$, that is, the components where the $\eps$-regularization is most influential. \\
\begin{algorithm}[h!]
	\caption{Active set monotone scheme}
	\begin{algorithmic}[1]
		\STATE Initialize ‎$\eps>0, x^0,\lambda^0, y^0=\Lambda x^0$. Set $n=0$.  
		\STATE{\textbf{repeat} \{outer loop\}}
		\STATE Let $\mathcal{A}_n=\{i \,:\, |B_i y^n_i+\lambda^n_i|\leq \mu_i\},$ $\mathcal{I}_n=\mathcal{A}_n^c$. Initialize $x^{0,n+1}=x^n, \lambda^{0,n+1}=\lambda^n$ and $y^{0,n+1}=\Lambda x^{0,n}$. Set $k=0$.
		\STATE{\textbf{repeat} \{inner loop\}}
		\STATE Solve for $x^{k+1,n+1}, \lambda_{\mathcal{A}_n}^{k+1,n+1}$
		\begin{equation}\label{sistemalg}
\left\{
\begin{array}{lll}
(A^*A+\Lambda_{\mathcal{I}_n}^*N_{\mathcal{I}_n}^{k,n+1} \Lambda_{\mathcal{I}_n}+\eta P) x^{k+1,n+1}+ \Lambda_{\mathcal{A}_n}^*\lambda^{k+1,n+1}_{\mathcal{A}_n}=A^*b\\
\Lambda_{\mathcal{A}_n} x^{k+1,n+1}=0 
\end{array}
\right.\,
\end{equation}
	   Set $y^{k+1,n+1}= \Lambda x^{k+1,n+1},  \lambda_{\mathcal{I}_n}^{k+1,n+1}=\frac{\beta p y^{k+1,n+1}_{\mathcal{I}_n}}{\max(\eps^{2-p},|y^{k+1,n+1}_{\mathcal{I}_n}|^{2-p})}$.
	   \STATE If $ y^{k+1,n+1}_{\mathcal{I}_n}$ is a singular point, go to $9$.
		\STATE Set $k=k+1$.
		\STATE{\textbf{until} the stopping criteria for the inner loop are fulfilled.}
		\STATE Set $n=n+1$;
		\STATE{\textbf{until} the stopping criteria for the outer loop are fulfilled.}
		\STATE Reduce $\eps$ and go to $3$.
	\end{algorithmic}
\end{algorithm}

In the case $\Lambda$ coincides with the identity the system \cref{optcondd} can be written as 
\begin{equation}\label{optcondrem}
\left\{
\begin{array}{lll}
x^{n+1}_i=0 \quad &\mbox{ if }\,\, i \in \mathcal{A}_n\\
(A_i,Ax^{n+1}-b)+\eta P_{ij}x^{n+1}_j+\frac{\beta p x^{n+1}_i}{\max(\eps^{2-p},|x^{n+1}_i|^{2-p})}=0 \quad & \mbox{ if }\,\, i \in \mathcal{I}_n.
\end{array}
\right.\,
\end{equation}
Note that in \cref{optcondrem} we coupled the first and the third equation in \cref{optcondd} and we eliminated the dual variable. 
The advantage  is that now we solve the second equation in \cref{optcondrem} only for the inactive components $x_{\mathcal{I}_n}$,  solving a  system of $|\mathcal{I}_n|$ equations, whereas  in   \cref{sistemalg}   we solve $n +|\mathcal{A}_n|$ equations.
Finally we remark that in the case $\Lambda$ coincides with the identity $\eps>0$ is fixed. In particular $\eps=\min_{i}\left(\frac{2 \beta (1-p)}{|A_i|_2^2}\right)^{\frac{1}{2-p}}$ accordingly to the lower bound on the inactive components given by \cref{lowerbound}.\\

\subsection{Sparsity in a time-dependent control problem}\label{control}
We test the active set monotone algorithm on the time-dependent control problem described in  \cref{controlone}, with the same discretization in space  and time ($\Delta x= \Delta t=\frac{1}{50}$) and target function $b$. Also the initialization of $x$ and the $\eps$-range are the same.
 In Tables $4$  we report the results of our tests for $p=.1$ and $\beta$ incrementally increasing by factor of $10$ from $10^{-3}$ to $1$. We report only the values for the second control $u_2$ since the first control $u_1$ is always zero. As expected, $|Du_2|^c_0$ increases and  $|Du_2|^p_p$ decreases when $\beta$ is increasing. Note that the number of iterations of the inner and outer cycle are both  small. \\
The algorithm was also tested for the same $p$ as in \cref{controlone}, that is   $p=.5$, for the same range of $\beta$ as in Table $4$. Comparing to the results achieved by \textbf{Algorithm $1$}, we obtained the same values for the  $\ell^0$-term for corresponding values of $\beta$ and a considerably smaller residue within a significantly fewer number of inner iterations.\\
Finally we note that if $\Lambda=I$ the number of inner iterations is even smaller, that is, $6$ on the average.

\begin{table}[tbhp]
	\caption{Sparsity in a time-dependent control problem,  $ p=.1$, mesh size $h=\frac{1}{50}$. Results obtained by \textbf{Algorithm $3$}.} \label{controltable2} 
	\centering
	\begin{tabular}{|l|c|c|c|c|}
		\hline
		{\bf $\beta$ }   &$10^{-3} $& $10^{-2}$& $10^{-1}$& $1$\\
		\hline
		no. of outer iterates   & 1           &1  &  4    &1\\
		\hline
		no. of inner iterates   &20          & 20  & 30    &20\\
		\hline
		{\bf $|Du_2|^c_0$ }   &95         &95  &98  &100\\
		\hline
		{\bf $|Du_2|^p_p$ }  &18    &17 &14 &   0 \\
		\hline
		$\mbox{ Residue }$ &$10^{-15}$& $10^{-15}$    & $10^{-14}$  &   $10^{-16}$\\
		\hline
	\end{tabular}
\end{table}

\subsection{Compressive sensing approach for microscopy  image reconstruction}\label{mimrec}
In this subsection we present an application of the active set monotone scheme to  compressive sensing for microscopy image reconstruction. 
We focus on the STORM  (stochastic optical reconstruction microscopy) method, which is based on stochastically switching and high-precision detection of single molecules  to achieve an image resolution beyond the diffraction limit. The literature on the STORM has  been intensively increasing, see e.g. \cite{RBZ}, \cite{BPS} \cite{HGM}, \cite{HBZ}.  The STORM  reconstruction process consists in a   series of imaging cycles. In each cycle  only a fraction of the
fluorophores in the field of view are switched on (stochastically), such that each of the active fluorophores is optically resolvable from the rest, allowing the position of these fluorophores to be determined with high accuracy. 
 Despite the advantage of obtaining
sub-diffraction-limit spatial resolution, in these single molecule detection-based techniques such as STORM, the time to acquire a super-resolution image is
limited by the maximum density of fluorescent emitters that can be accurately localized per imaging frame, see e.g. \cite{SGGB}, \cite{JSHZ}, \cite{NLB}. In order to get at the same time better resolution and higher emitter density per imaging frame, compressive sensing methods based on $l^1$  techniques have been recently applied, see e.g. \cite{ZZEH}, \cite{BMYX}, \cite{GSC} and the references therein.   In the following, we propose a similar approach based on our  $l^p$ with  $p<1$ methods.  We mention that  $l^p$ with $0<p\leq 1$ techniques based on a concave-convex regularizing procedure, and hence different from ours, are  used in \cite{KMCUMJY}.

To be more specific, each single frame reconstruction can be achieved  by solving the following constrained-minimization problem: 
\begin{equation}\label{minprobimage}
 \min_{x \in \R^n} |x|^p_{p} \quad \mbox{ such that } \, \,|A x-b|_2 \leq \eps,
\end{equation}
where $p \in (0,1]$, $x$ is the up-sampled, reconstructed image, $b$ is the experimentally observed image, and $A$ is the impulse reponse (of size $m\times n$, where $m$ and $n$ are the numbers of pixels in $b$ and $x$, respectively). $A$ is usually called the  point spread function (PSF) and  describes the response of an imaging system to a point source or point object.   The inequality constraint on the $\ell^2$-norm allows some inaccuracy in the image reconstruction to accommodate the statistical corruption of the image by noise \cite{ZZEH}. Solving problems as \cref{minprobimage} is referred to as compressed sensing in the literature of miscroscopy imaging. Indeed, in the basic compressed sensing problem, an under-determined, sparse signal vector is reconstructed  from a noisy measurement in a basis in which the signal is not sparse.  In the compressed sensing approach to microscopy image reconstruction, the sparse basis is a high resolution grid, in which fluorophore locations are presented,  while the noisy measurement basis is the lower resolution camera pixels, on which fluorescence signal are detected experimentally. In this framework, the optimally reconstructed image is the one that contains the fewest number of fluorophores but reproduces the measured image on the camera to a given accuracy (when convolved with the optical impulse reponse). \\
We  reformulate problem \cref{minprobimage}  as: 
\begin{equation}\label{minprobimagebeta}
\min_{x \in \R^n} \frac{1}{2}|Ax-b|^2_2+\beta |x|^p_p
\end{equation}
and we solve  \cref{minprobimagebeta} by applying  \textbf{Algorithm $3$}.
Note that we may consider \cref{minprobimagebeta} arising from \eqref{minprobimage} with $\beta$ related to the reciprocal of the Lagrange multiplier associated to the inequality constraint $|A x-b|_2 \leq \eps$. \\
%
%
First we tested the procedure for same resolution images, in particular the conventional and the true images are both $128\times 128$ pixel images.  
 Then the algorithm was tested in the case of a $16\times 16$ pixel conventional image and a $128 \times 128$ true image. 
The values for the impulse reponse $A$ and the measured data $b$ were chosen according to the literature, in particular $A$ was taken as the Gaussian PSF matrix with variance $\sigma=8$ and  size $3\times \sigma=24$, and $b$ was simulated by convolving the impulse reponse $A$ with a random $0$-$1$ mask over the image  adding a white random noise so that the signal to noise ratio is  $.01$. \\
 We carried out several tests with the same data for different values of  $p,\beta$. We report only our results for  $p=.1$ and $\beta=10^{-6}, \beta=10^{-9}$ for the same and the  different resolution case respectively, since for these values the best reconstructions were achieved. The number of single frame reconstructions carried out to get the full reconstruction was $5, 10$ for the same, different resolution case, respectively. \\
In order to measure the performance of our algorithm, we plot a graphic of the average over six recoveries of the location recovery and the exact recovery (up to a certain tolerance) against the noise. Note that in  compressed sensing  these quantities are typically used as a measure of the efficacy of the reconstruction method, see for example \cite{DP} (where, under certain conditions, a linear decay with respect to the noise is proven) and \cite{C}. \\
The first test is carried out for a sparse $0$-$1$ cross-like image.   The STORM reconstructions are presented in Figures $5,6$   for the same and different resolution case, respectively. In Figures $7$ the plots of the location and exact recovery  are shown in the case of different resolution. Similar plots are obtained in the same resolution case.
Note that our algorithm can recover quite well the location of the emitters. Also, the location and intensity of the emitters decay linearly with respect to the noise level, in line with the result of \cite{DP}. In particular, for small  noise both the  recoveries are very near to $n^2=16384$, that is, the exact recovery is $16240, 16243$ and the location  is $16384, 16360$ for the same and the different resolution case, respectively. We observe also that the values of the location recovery are higher than the exact recovery for small values of the noise, as expected. \\
A second test on a non sparse standard phantom image is carried out. In Figure $8$ we show the reconstruction in the case of same resolution images. Note that a high percentage of emitters is correctly localized and the boundaries of the image are well-recovered. Also in this case the location and exact recoveries show a linear decay with respect to the noise.\\
In Tables $5,6$ we report the number of iterations needed for each single frame reconstruction.  For the cross image in the different resolution case (Table $5$),  the number of iterations is averagely $100, 164$ for the outer cycle and inner cycle, respectively. Note that for the phantom in the same resolution case (Table $6$) the number of iterations is lower, that is averagely $7.2, 9.8$ for the outer cycle and inner cycle, respectively. The numbers of iterations for the cross image in case of same resolution  are comparable to the ones of Table $5$. As shown in the third line of each tables, the residue is always less than or equal to $10^{-8}$.\\
We compared our results with the ones obtained by the FISTA in the same situations and same values of the parameters as described above. Figure $9$ shows a comparison between   the number of surplus and missed emitters recovered (Error+, Error- respectively)  by \textbf{Algorithm $3$} and the FISTA in the case of the cross image and different resolution.
We remark that the levels of the location and exact recoveries achieved by the FISTA are lower than the ones obtained by \textbf{Algorithm $3$}, at least for values of the noise near $.01$. In particular, by the FISTA the Error+  is always above $410$, whereas by \textbf{Algorithm $3$} is zero for small value of the noise.   On the other hand, FISTA is faster than our algorithm (as expected, since our algorithm solves a nonlinear equation for each minimization problem.)  
\begin{figure}[h!]
	\centering
	\subfloat[\tiny{Real distribution}]{
		\includegraphics[width=0.35\textwidth]{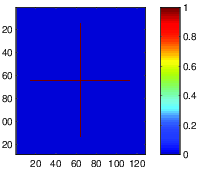}}
		\hspace{0.2cm}
		\subfloat[\tiny{Simulated single frame image}]{\includegraphics[width=0.36\textwidth]{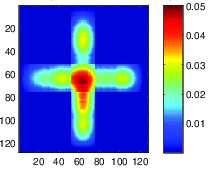}}
		\hspace{0.2cm}
		\subfloat[\tiny{Single frame sparse reconstruction}]{\includegraphics[width=0.35\textwidth]{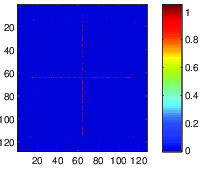}}
		\hspace{0.2cm}
		\subfloat[\tiny{Full STORM sparse reconstruction}]{\includegraphics[width=0.35\textwidth]{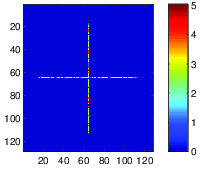}}
		\hspace{0.2cm}
		\caption{A STORM reconstruction procedure, same resolution, $p=.1, \beta=10^{-6}$. Results obtained by \textbf{Algorithm $3$}.}
\end{figure}
\vspace{1cm}


		\begin{figure}[!htbp]
	\centering
	\subfloat[\tiny{Real distribution}]{
		\includegraphics[width=0.35\textwidth]{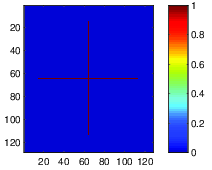}}
		\hspace{0.2cm}
		\subfloat[\tiny{Simulated single frame image}]{\includegraphics[width=0.36\textwidth]{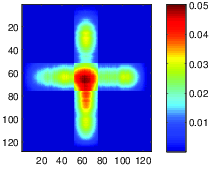}}
		\hspace{0.2cm}
		\subfloat[\tiny{Single frame sparse reconstruction}]{\includegraphics[width=0.35\textwidth]{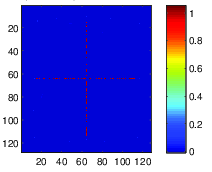}}
		\hspace{0.2cm}
		\subfloat[\tiny{Full STORM sparse reconstruction}]{\includegraphics[width=0.35\textwidth]{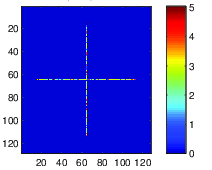}}
		\hspace{0.2cm}
		\caption{A STORM reconstruction from a $16x16$ pixel image, different resolution, $p=.1, \beta=10^{-9}$. Results obtained by \textbf{Algorithm $3$}.}
\end{figure}



\begin{figure}[h!]
\centering
\subfloat
{
\includegraphics[height=4.5cm, width=4.5cm]{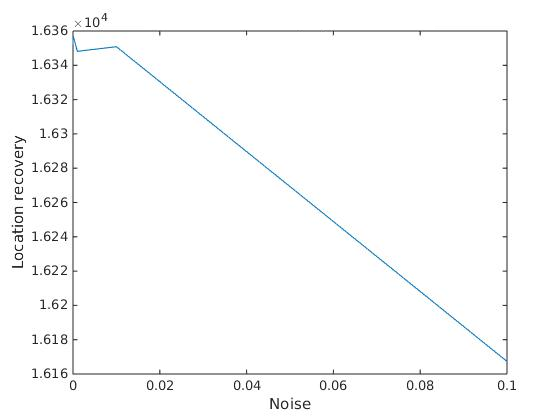}
}
\hspace{0.6cm}
\subfloat
{
\includegraphics[height=4.5cm, width=4.5cm]{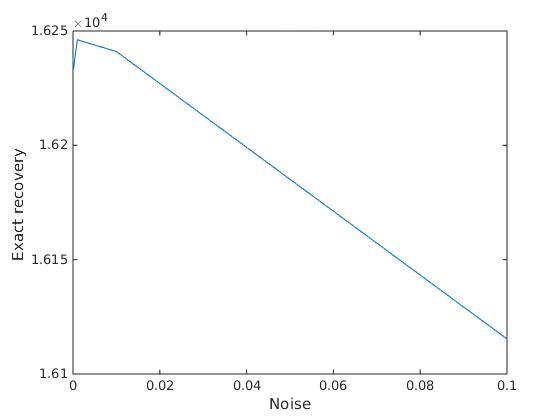}
}
\caption{Left: location recovery. Right: exact recovery. Cross image, different resolution, $p=.1, \beta=10^{-9}$.  Results obtained by \textbf{Algorithm $3$}.
\vspace{1.5cm}}
\end{figure}


\begin{figure}[h!]\label{figim3}
	\centering
	\subfloat[\tiny{Real distribution}]{
		\includegraphics[width=0.35\textwidth]{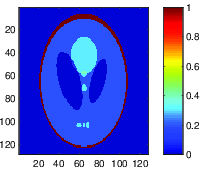}}
		\hspace{0.2cm}
		\subfloat[\tiny{Simulated single frame image}]{\includegraphics[width=0.35\textwidth]{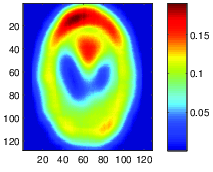}}
		\hspace{0.2cm}
		\subfloat[\tiny{Single frame sparse reconstruction}]{\includegraphics[width=0.35\textwidth]{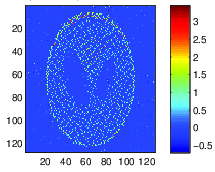}}
		\hspace{0.2cm}
		\subfloat[\tiny{Full STORM sparse reconstruction}]{\includegraphics[width=0.35\textwidth]{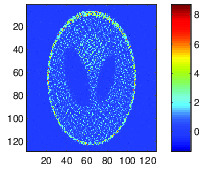}}
		\hspace{0.2cm}
		\caption{A STORM reconstruction procedure, same resolution, $p=.1, \beta=10^{-6}$. Results obtained by \textbf{Algorithm $3$}.}
\end{figure}
\begin{figure}[h!]
\centering
\subfloat[\tiny{$p=.1, \beta=10^{-6}$ by \textbf{Algorithm $3$}}]{
\includegraphics[height=5.5cm, width=5.5cm]{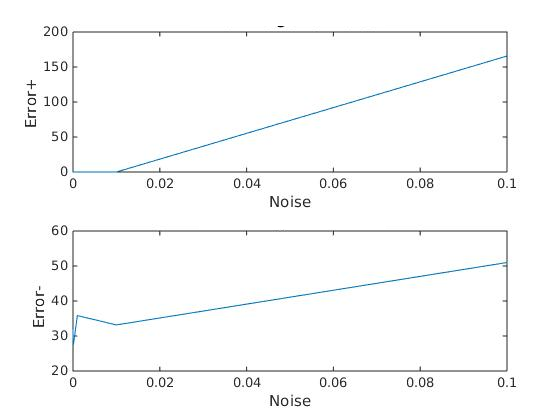}
}
\subfloat[\tiny{ $p=.1, \beta=10^{-4}$ by FISTA}]{
\includegraphics[height=5.5cm, width=5.5cm]{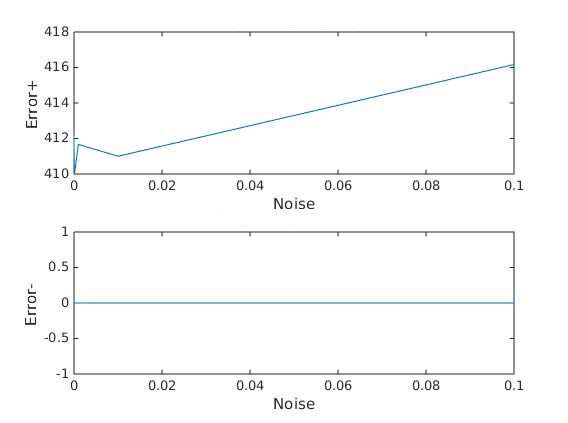}
}
\caption{Graphics of Error+ (surplus of emettitors), Error- (missed emettitors) against noise.}
\end{figure}


	\begin{table}[h!]
	\caption{Number of outer and inner iterations (ItOut, ItIn)  and residue (Res) for eache single frame (Fr). Cross image with different resolution,  $p=.1, \beta=10^{-9}$. Results obtained by \textbf{Algorithm $3$}.} \label{itdres} 
	\begin{tabular}{|l|c|c|c|c|c|c|c|c|c|c|}
	\hline\noalign{\smallskip}
		Fr &1&2&3&4&5&6&7&8&9&10\\
		\noalign{\smallskip}\hline\noalign{\smallskip}
		ItOut    &100   & 98&100&100&100&100&100&85&100&100   \\
		ItIn       &147     & 190&144&184&145&186&146&187&145&165  \\
		Res &$10^{-8}$ &$10^{-8}$ &$10^{-8}$&$10^{-8}$&$10^{-9}$&$10^{-8}$&$10^{-8}$&$10^{-8}$&$10^{-8}$&$10^{-8}$\\
		\noalign{\smallskip}\hline
	\end{tabular}
\end{table}

\begin{table}[h!]
\caption{Number of outer and inner iterations (ItOut, Itin) and residue (Res) for each single frame (Fr). Phantom image with same resolution,  $p=.1, \beta=10^{-6}$. Results obtained by \textbf{Algorithm $3$}.} \label{itdres} 
	\begin{tabular}{|l|c|c|c|c|c|}
		\hline\noalign{\smallskip}
		Fr &1&2&3&4&5\\
		\noalign{\smallskip}\hline\noalign{\smallskip}
		ItOut       &6&11&7&6&6      \\
		ItIn      &9&14&12&7&7     \\
		Res & $10^{-8}$& $10^{-10}$ &$10^{-12}$ &$10^{-8}$&$10^{-8}$\\
		\noalign{\smallskip}\hline
	\end{tabular}
\end{table}

\end{document}